\documentclass{amsart}
\usepackage{latexsym,amssymb,graphicx}
\usepackage{epsfig}
\usepackage{color}
\usepackage{mathrsfs}
\usepackage{amsmath,amscd}
\usepackage[all]{xy}
\newtheorem{theorem}{Theorem}[section]
\newtheorem{lemma}[theorem]{Lemma}
\newtheorem{corollary}[theorem]{Corollary}
\newtheorem{proposition}[theorem]{Proposition}

\theoremstyle{definition}
\newtheorem{definition}[theorem]{Definition}
\newtheorem{example}[theorem]{Example}
\theoremstyle{remark}
\newtheorem{remark}[theorem]{Remark}
\numberwithin{equation}{section}

\newcommand{\mm}{\hspace{.5mm}}
 
\newcommand{\Hom}{\mathrm{Hom}}

\newcommand{\sA}{{\mathcal A}}

\newcommand{\sC}{{\mathcal C}}
\newcommand{\sD}{{\mathcal D}}
\newcommand{\sE}{{\mathcal E}}
\newcommand{\sF}{{\mathcal F}}

\newcommand{\sH}{{\mathcal H}}
\newcommand{\sI}{{\mathcal I}}
\newcommand{\sJ}{{\mathcal J}}
\newcommand{\sK}{{\mathcal K}}

\newcommand{\sM}{{\mathcal M}}
\newcommand{\sN}{{\mathcal N}}

\newcommand{\sP}{{\mathcal P}}

\newcommand{\sR}{{\mathcal R}}
\newcommand{\sS}{{\mathcal S}}

\newcommand{\A}{{\mathbb A}}

\newcommand{\G}{{\mathbb G}}

\newcommand{\bN}{{\mathbb N}}

\newcommand{\mS}{{\mathbb S}}

\newcommand{\Z}{{\mathbb Z}}

\renewcommand{\phi}{\varphi}

\renewcommand{\1}{{\mathbf{1}}}

\newcommand{\Div}{{\rm Div}}

\newcommand{\Spec}{\operatorname{Spec}}

\newcommand{\Char}{\operatorname{char}}

\newcommand{\id}{{\operatorname{id}}}

\newcommand{\Sch}{{\operatorname{\mathbf{Sch}}}}

\newcommand{\op}{{\text{\rm op}}}

\newcommand{\del}{\partial}

\newcommand{\Sm}{{\mathbf{Sm}}}

\newcommand{\hocolim}{\operatornamewithlimits{hocolim}}
\renewcommand{\lim}{\operatornamewithlimits{lim}}
\newcommand{\colim}{\operatornamewithlimits{colim}}

\newcommand{\Ho}{{\mathbf{Ho}\,}}

\newcommand{\sq}{\square}
 \newcommand{\Ab}{{\mathbf{Ab}}}

\newcommand{\SH}{{\operatorname{\sS\sH}}}

\newcommand{\eff}{{\mathop{eff}}}

\newcommand{\DM}{{DM}}

\newcommand{\ds}{{/\kern-3pt/}}

\newcommand{\hocofib}{{\mathop{\rm{hocofib}}}}

\newcommand{\EM}{{{EM}_{\A^1}}}

\newcommand{\Mot}{{\mathop{Mot}}}

\newcommand{\Mod}{{\operatorname{Mod}}}

\renewcommand{\:}{\kern-1.5pt:\kern-1.5pt}

\newcommand{\gr}{\mathbf{Gr}}

\newcommand{\MGL}{{\mathop{MGL}}}

\newcommand{\BP}{{\mathop{BP}}}
\renewcommand{\L}{\mathbb{L}}

\newcommand{\lci}{l.\,c.\,i.~}
\newcommand{\SP}{{\mathbf{SP}}}
\newcommand{\ph}{\text{ph}}
\newcommand{\N}{\mathbb{N}}

\begin{document}

\title[Quotients of $\MGL$, their slices and their geometric parts]{Quotients of $\mathbf{MGL}$, their slices and their geometric parts}

\author{Marc Levine}
\address{
Universit\"at Duisburg-Essen\\
Fakult\"at Mathematik\\
Thea-Leymann-Stra{\ss}e 9\\
45127 Essen\\
Germany}
\email{marc.levine@uni-due.de}
\author{Girja Shanker Tripathi}
\address{
Universit\"at Duisburg-Essen\\
Fakult\"at Mathematik\\
Thea-Leymann-Stra{\ss}e 9\\
45127 Essen\\
Germany}
\email{tripathigirja@gmail.com}

\thanks{Both authors wish to thank the Humboldt Foundation for financial support}

\subjclass{Primary 14C25, 19E15; Secondary 19E08 14F42, 55P42}

\begin{abstract}
Let $x_1, x_2,\ldots$ be a system of homogeneous polynomial generators for the Lazard ring $\L^*=MU^{2*}$ and let $\MGL_S$ denote Voevodsky's algebraic cobordism spectrum in the motivic stable homotopy category over a base-scheme $S$ \cite{VoevICM}. Take $S$ essentially smooth over a field $k$. Relying on the Hopkins-Morel-Hoyois isomorphism \cite{Hoyois}  of the  0th slice $s_0\MGL_S$ for Voevodsky's slice tower with  $\MGL_S/(x_1, x_2,\ldots)$   (after inverting the characteristic of $k$), Spitzweck \cite{Spitzweck10} computes the remaining slices of $\MGL_S$ as $s_n\MGL_S=\Sigma^n_TH\Z\otimes \L^{-n}$ (again, after inverting the characteristic of $k$). We apply Spitzweck's method to compute the slices of a quotient spectrum $\MGL_S/(\{x_i:i\in I\})$ for $I$ an arbitrary subset of $\N$, as well as the mod $p$ version $\MGL_S/(\{p, x_i:i\in I\})$ and localizations with respect to a system of homogeneous elements in $\Z[\{x_j:j\not\in I\}]$. In case $S=\Spec k$, $k$ a field of characteristic zero, we apply this to show that for $\sE$ a localization of a quotient of $\MGL$ as above, there is a natural isomorphism   for the theory with support
\[
\Omega_*(X)\otimes_{\L^{-*}}\sE^{-2*,-*}(k)\to \sE^{2m-2*, m-*}_X(M)
\]
for $X$ a closed subscheme of a smooth quasi-projective $k$-scheme $M$, $m=\dim_kM$.

\end{abstract}

\setcounter{tocdepth}{1}

\maketitle

\tableofcontents

\section*{Introduction} This paper has a two-fold purpose. We consider Voevodsky's slice tower on the motivic stable homotopy category $\SH(S)$ over a base-scheme $S$ \cite{VoevSlice}. For $\sE$ in $\SH(S)$, we have the $n$th layer $s_n\sE$ in the slice tower for $\sE$. Let  $\MGL$ denote Voevodsky's algebraic cobordism spectrum in $\SH(S)$  \cite{VoevICM} and let $x_1, x_2,\ldots$
be a system of homogeneous polynomial generators for the Lazard ring $\L_*$. Via the classifying map for the formal group law for $\MGL$, we may consider $x_i$ as an element of $\MGL^{2i,i}(S)$, and thereby as a map $x_i:\Sigma^{2i,i}\MGL\to \MGL$, giving the quotient $\MGL/(x_1, x_2,\ldots)$. Spitzweck \cite{Spitzweck10} shows how to build on the Hopkins-Morel-Hoyois
isomorphism \cite{Hoyois}
\[
\MGL/(x_1, x_2,\ldots)\cong   s_0\MGL
\]
to compute all the slices $s_n\MGL$ of $\MGL$. Our first goal here is to extend
Spitzweck's method to handle quotients of $\MGL$  by a subset of  $\{x_1, x_2,
\ldots\}$, as well as localizations with respect to a system of homogeneous elements in 
the ring generated by the remaining variables; we also consider quotients of such spectra by an integer. Some
of these spectra are Landweber  exact, and the slices are thus computable by the
results of Spitzweck on the slices of Landweber exact spectra
\cite{Spitzweck12}, but many of these, such as the truncated Brown-Peterson
spectra or Morava $K$-theory, are not. 

The second goal is to extend results of \cite{DaiLevine, LevineComp,
LevineExact}, which consider the ``geometric part'' $X\mapsto \sE^{2*,*}(X)$ of
the bi-graded cohomology defined by an oriented weak commutative ring
$T$-spectrum $\sE$ and raise the question: is the
classifying map
\[
\sE^*(k)\otimes_{\L^*}\Omega^*\to \sE^*
\]
 an isomorphism of oriented cohomology theories, that is, is the theory $\sE^*$
 a theory of  {\em rational type}  in the sense of Vishik \cite{Vishik}?
Starting with the case $\sE=\MGL$, discussed in \cite{LevineComp}, which
immediately yields the Landweber exact case,  we have answered this   affirmatively for  ``slice
effective'' algebraic $K$-theory in \cite{DaiLevine}, and extended to the case of
slice-effective covers of a Landweber exact theory in \cite{LevineExact}. In this paper,
we use our computation of the slices of a quotient of $\MGL$ to  show that the classifying map is an isomorphism for
the quotients and localizations of $\MGL$ described above. 

The paper is organized as follows: in \S\ref{general quotients} and \S\ref{sec:SlicesModuleSpectra}, we axiomatize
Spitzweck's method from \cite{Spitzweck10} to a more general setting. In \S\ref{general quotients} we give a description of quotients in a suitable symmetric monoidal model category in terms of a certain homotopy colimit. 
In \S\ref{sec:SlicesModuleSpectra}  we begin by recalling some basic facts and the slice tower and its construction. We then apply the results of \S\ref{general quotients} to the category of $\sR$-modules in a symmetric
monoidal model category (with some additional technical assumptions), developing
a method for computing the slices of an $\sR$-module $\sM$, assuming that $\sR$
and $\sM$ are effective and that the 0th slice $s_0\sM$ is of the form
$\sM/(\{x_i:i\in I\})$ for some collection  $\{ [x_i]\in \sR^{-2d_i, -d_i}(S),
d_i<0\}$ of elements in $\sR$-cohomology of the base-scheme $S$; see
theorem~\ref{thm:slices of effective spectra}. We also discuss localizations of
such $\sR$-modules and the mod $p$ case (corollary~\ref{cor:slices of localized
spectra} and corollary~\ref{cor:slices of quotient spectra}). We discuss the
associated slice spectral sequence for such $\sM$
and its convergence properties in \S\ref{sec:SliceSS}, and apply these
results to our examples of interest: truncated Brown-Peterson spectra, Morava
$K$-theory and connective Morava $K$-theory,  as well as the Landweber exact examples, the  Brown-Peterson spectra $\BP$ and the Johnson-Wilson spectra $E(n)$,  in
\S\ref{sec:SlicesTruncatedBPSpectra}.

The remainder of the paper discusses the classifying map from algebraic
cobordism $\Omega_*$ and proves our results on the rationality of certain
theories.  This is essentially taken from \cite{LevineExact}, but we need to
deal with a technical problem, namely,  that it is not at present clear if the
theories $[\MGL/(\{x_i:i\in I\})]^{2*,*}$ have a multiplicative structure. For
this reason, we extend the setting used in \cite{LevineExact} to theories that
are modules over ring-valued theories. This extension is taken up in
\S\ref{sec:modules} and we apply this theory to quotients and localizations of
$\MGL$ in \S\ref{sec:Apps}.

\section{Quotients and homotopy colimits  in a model category}
\label{general quotients}

In this section we consider certain quotients  in a model category  and give a description of these quotients as a homotopy colimit (see proposition~\ref{prop:zeroth slice}). This is an abstraction of the methods developed in \cite{Spitzweck10} for computing the slices of $\MGL$.

Let $(\sC, \otimes, 1)$ be a closed symmetric monoidal simplicial pointed model
category with cofibrant unit $1$. We assume that $1$ admits a fibrant
replacement $\alpha:1\to\1$ such that $\1$ is a $1$-algebra in $\sC$, that is, there is an
associative multiplication map $\mu_\1:\1\otimes \1\to \1$ such that
$\mu_1\circ(\alpha\otimes\id)$ and 
$\mu_1\circ(\id\otimes\alpha)$ are the respective multiplication isomorphisms $1\otimes
\1\to \1$, $\1\otimes1\to\1$.

For a cofibrant object $T$ in $\sC$, the  map $T\cong  T\otimes
1\xrightarrow{\id\otimes\alpha}  T\otimes\1$ is a cofibration and weak
equivalence. Indeed, the functor $T\otimes(-)$ preserves cofibrations, and also maps
that are both a cofibration and a weak equivalence, whence the assertion.

\begin{remark} \label{rem:Fibrant} We will be applying the results of this
section to the following situation: $\sM$ is a cofibrantly  generated symmetric
monoidal simplicial model category satisfying the monoid axiom \cite[definition
3.3]{SchwedeShipley}, $\sR$ is a commutative monoid in $\sM$, cofibrant in $\sM$
and $\sC$ is the category of $\sR$-modules in $\sC$, with model structure as in 
\cite[\S4]{SchwedeShipley}, that is, a map is a fibration or a weak equivalence
in $\sC$ if and only if it is so as a map in $\sM$, and cofibrations are determined
by the LLP with respect to acyclic fibrations. By  \cite[theorem
4.1(3)]{SchwedeShipley}, the category $\sR$-Alg of monoids in $\sC$ has
the structure of a cofibrantly generated model category, with fibrations and
weak equivalence those maps which become a fibration or weak equivalence in
$\sM$, and each cofibration in $\sR$-Alg is a cofibration in $\sC$. The unit 1
is $\sC$ is just $\sR$ and we may take $\alpha:1\to\1$ to be a fibrant
replacement in $\sR$-Alg.
\end{remark}

Let $\{x_i:T_i \to \1\ |\ i\in I\}$ 
be a set of maps with cofibrant sources $T_i$. We assign each $T_i$ an integer
 degree $d_i>0$.

Let $\1/(x_i)$ be the homotopy cofiber (i.e., mapping cone) of the map 
$x_i:\1\otimes T_i \to \1$ and let $p_i:\1\to \1/(x_i)$ be the canonical map.

Let $A=\{i_1,\ldots, i_k\}$ be a finite subset of $I$ and define $\1/(\{x_i :
i\in A\})$ as 
\[
\1/(\{x_i : i\in A\}):=\1/(x_{i_1})\otimes\ldots\otimes \1/(x_{i_k}).
\]
Of course, the object $\1/(\{x_i : i\in A\})$ depends on a choice of ordering of
the elements in $A$, but only up to a canonical symmetry isomorphism. We could
for example fix the particular choice by fixing a total order on $A$ and taking
the product in the proper order.The canonical maps $p_i$, $i\in I$ composed with
the map $1\to \1$ give rise to the canonical map
\[
p_I:1\to \1/(\{x_i : i\in A\})
\]
defined as the composition
\[
1\xrightarrow{\mu^{-1}}1^{\otimes k}\to \1^{\otimes
k}\xrightarrow{p_{i_1}\otimes\ldots\otimes p_{i_k}} \1/(\{x_i : i\in A\}).
\]

For finite subsets $A\subset B\subset I$, define the map
\[
\rho_{A\subset B}:\1/(\{x_i : i\in A\})\to \1/(\{x_i : i\in B\})
\]
as the composition
\begin{multline*}
\1/(\{x_i : i\in A\})\xrightarrow{\mu^{-1}}\1/(\{x_i : i\in A\})\otimes
1\\\xrightarrow{\id\otimes p_{B\setminus A}}
\1/(\{x_i : i\in A\})\otimes \1/(\{x_i : i\in B\setminus A\})\cong \1/(\{x_i :
i\in B\}).
\end{multline*}
where the last isomorphism is again the symmetry isomorphism.

Because $\sC$ is a symmetric monoidal category with unit $1$, we have a
well-defined functor from the category $\sP_\text{fin}(I)$ of finite subsets of
$I$ to $\sC$:
\[
\1/(-):\sP_\text{fin}(I)\to \sC
\]
sending $A\subset I$ to $\1/(\{x_i : i\in A\})$ and sending each inclusion
$A\subset B$ to $\rho_{A\subset B}$.

\begin{definition} The object $\1/(\{x_i : i\in I\})$ of $\sC$ is defined by
\[
\1/(\{x_i\})= \hocolim_{A\in \sP_\text{fin}(I)} \1/(\{x_i : i\in A\}).
\]
More generally, for $M\in \sC$, we define $M/(\{x_i : i\in I\})$ as
\[
M/(\{x_i : i\in I\}):=\1/(\{x_i : i\in I\})\otimes QM,
\]
where $QM\to M$ is a cofibrant replacement for $M$. In case the index set $I$ is
understood, we often write these simply as $\1/(\{x_i \})$ or $M/(\{x_i \})$.
\end{definition}

\begin{remark} 1. The object $\1/(x_i)$ is  
cofibrant and hence the objects  $\1/(\{x_i : i\in A\})$ are cofibrant for all
finite sets $A$. As a pointwise cofibrant diagram has cofibrant homotopy colimit
\cite[corollary 14.8.1, example 18.3.6, corollary 18.4.3]{Hirschhorn},
$\1/(\{x_i : i\in I\})$ is cofibrant. Thus $M/(\{x_i : i\in I\}):=\1/(\{x_i :
i\in I\})\otimes QM$ is also cofibrant.\\
2. We often select a single cofibrant object $T$ and take $T_i:=T^{\otimes d_i}$
for certain integers $d_i>0$. As $T$ is cofibrant, so is $T^{\otimes d_i}$. In
this case we set $\deg T=1$, $\deg T^{\otimes d_i}=d_i$.
\end{remark}

We let $[n]$ denote the set $\{0,\ldots, n\}$ with the standard order and $\Delta$ the category with objects $[n]$, $n=0, 1,\ldots,$ and morphisms the order-preserving maps of sets. For a small category $A$ and a functor $F:A\to \sC$, we let $\hocolim_AF_*$
denote the standard simplicial object of $\sC$ whose geometric realization is
$\hocolim_AF$, that is 
\[
\hocolim_AF_n=\coprod_{\sigma:[n]\to A}F(\sigma(0)).
\]

\begin{lemma} Let $\{x_i:T_i\to \1 : i\in I_1\}$, $\{x_i:T_i\to \1 : i\in I_2\}$
be two sets of maps in $\sC$, with cofibrant sources $T_i$. Then there is a
canonical isomorphism
\[
 \1/(\{x_i: i\in I_1\amalg I_2\})\cong \1/(\{x_i: i\in I_1\})\otimes \1/(\{x_i:
i\in I_2\}).
\]
\end{lemma}

\begin{proof} The category $\sP_\text{fin}(I_1\amalg I_2)$ is clearly equal to
$\sP_\text{fin}(I_1)\times \sP_\text{fin}(I_2)$. 
For functors $F_i:\sA_i\to \sC$, $i=1,2$,  $[\hocolim_{\sA_1\times
\sA_2}F_1\otimes F_2]_*$ is the diagonal simplicial space associated to the
bisimplicial space $(n,m)\mapsto [\hocolim_{\sA_1}F_1]_n\otimes
 [\hocolim_{\sA_2}F_2]_m$. Thus 
 \[
\hocolim_{\sA_1\times \sA_2}F_1\otimes F_2\cong
\hocolim_{\sA_2}[\hocolim_{\sA_1}F_1]\otimes F_2.
\]
This gives us the isomorphism
\begin{align*}
\1/&(\{x_i: i\in I_1\amalg I_2\})\\&=\hocolim_{(A_1, A_2)\in
\sP_\text{fin}(I_1)\times \sP_\text{fin}(I_2)} 
\1/(\{x_i : i\in A_1\})\otimes \1/(\{x_i : i\in A_2\})\\
&\cong \hocolim_{A_1\in \sP_\text{fin}(I_1)} 
\1/(\{x_i : i\in A_1\})\otimes\hocolim_{A_2\in \sP_\text{fin}(I_2)}  \1/(\{x_i :
i\in A_2\})\\
&=\1/(\{x_i: i\in I_1\})\otimes \1/(\{x_i: i\in I_2\})
\end{align*}
\end{proof}

\begin{remark}\label{rem:DoubleQuot}  Via this lemma, we have the isomorphism
for all $M\in \sC$,
\[
M/(\{x_i: i\in I_1\amalg I_2\})\cong (M/(\{x_i: i\in I_1\})/(\{x_i: i\in I_2\}).
\]
\end{remark}

Let $\sI$ be the category of formal monomials in 
$\{x_i\}$, that is, the category of maps $N:I\to \bN$, $i\mapsto N_i$,  such that $N_i=0$ for all
but finitely many $i\in I$, and with a unique map $N\to M$ if $N_i\ge M_i$ for
all $i\in I$. As usual, the monomial in the $x_i$ corresponding to a given $N$ is
$\prod_{i\in I}x_i^{N_i}$, written $x^N$.  The index $N=0$, corresponding to $x^0=1$, is
the final object of $\mathcal{I}$. 

Take an $i\in I$. For $m> k\ge0$ integers, define the map
\[
\times x_i^{m-k}:\1\otimes T_i^{\otimes m}\to \1\otimes T_i^{\otimes k}
\]
as the composition 
\[
\1\otimes T_i^{\otimes m}=\1\otimes T_i^{\otimes m-k}\otimes  T_i^{\otimes k}
\xrightarrow{\id_\1\otimes x_i^{\otimes m-k}\otimes \id_{ T_i^{\otimes
k}}}\1^{\otimes m-k+1}\otimes  T_i^{\otimes k}
\xrightarrow{\mu\otimes\id}\1\otimes  T_i^{\otimes k}.
\]
In case $k=0$, we use $\1$ instead of $\1\otimes 1$ for the target; we define
$\times x^0$ to be the identity map. The associativity of the maps $\mu_\1$ shows
that $\times x_i^{m-k}\circ \times x_i^{n-m}=\times x_i^{n-k}$, hence the maps
$\times x_i^n$ all commute with each other. 

Now suppose we have a monomial in the $x_i$; to simplify the notation, we write
the indices occurring in the monomial as $\{1,\ldots, r\}$ rather than
$\{i_1,\ldots, i_r\}$. This gives us the monomial
$x^N:=x_1^{N_1}\cdot\ldots\cdot x_r^{N_r}$.  Define  
\[
T_*^N:=\1\otimes T_1^{\otimes N_1}\otimes\ldots\otimes \1\otimes T_r^{\otimes
N_r}\otimes\1;
\]
in case $N_i=0$,  we replace    $\ldots\otimes\1\otimes 1\otimes \1\otimes
T_{i+1}^{\otimes M_{i+1}}\otimes\ldots$ with  
 $\ldots \otimes\1\otimes T_{i+1}^{\otimes M_{i+1}}\otimes\ldots$, and we set
$T_*^0:=\1$.

Let $N\to M$ be a map in $\sI$, that is $N_i\ge M_i\ge0$ for all $i$. We again
write the relevant index set as $\{1,\ldots, r\}$. Define the map
\[
\times x^{N-M}:  T_*^N\to   T_*^M
\]
as the composition
\[
T_*^N\xrightarrow{  \bigotimes_{j=1}^r \times x_j^{N_j-M_j}}
 \1\otimes T_1^{\otimes M_1}\otimes\ldots\otimes  \1 \otimes T_r^{\otimes M_r\otimes\1}
\xrightarrow{\mu_M}   T_*^{\otimes M};
\]
the map $\mu_M$ is a composition of $\otimes$-product of multiplication maps $\mu_\1:\1\otimes
\1\to \1$, with these occurring in those spots with $M_j=0$.  In case
$N_i=M_i=0$, we simply delete the term $\times x_i^0$ from the expression.

The fact that the maps $\mu_\1$ satisfy associativity yields  the relation
\[
\times x^{M-K}\circ \times x^{N-M}=\times x^{N-K}
\]
and thus the  maps $\times x^{N-M}$ all commute with each other.

Defining $\sD_x(N):=T_*^N$ and $\sD_x(N\to M)=\times x^{N-M}$ gives us the
$\sI$-diagram 
\[
\sD_x:\sI\to \sC.
\]

We  consider the following full subcategories of $\sI$. 
For a monomial
$M$  let  $\sI_{\geq M}$ denote the subcategory of monomials which are divisible
by $M$,
and for a positive integer $n$, recalling that we have assigned each $T_i$ a
positive integral degree $d_i$,  let   $\sI_{\deg \geq n}$ denote the
subcategory
of monomials of degree at least $n$, where the degree of 
$N:=(N_1,\ldots, N_k)$ is $N_1d_1+\cdots+N_kd_k$.  One defines similarly
the 
full subcategories $\sI_{> M}$ and $\sI_{\deg > n}$.

Let $\sI^\circ$ be the full subcategory of 
$\sI$ of  monomials $N\neq0$
and $\sI^\circ_{\le1}\subset \sI^\circ$ be the full subcategory 
of monomials $N$ for which $N_i\le 1$ for all $i$. 
We have the corresponding subdiagrams $\sD_x:\sI^\circ\to\sC$ and
$\sD_x:\sI^\circ_{\le1}\to\sC$
of $\sD_x$. For $J\subset I$ a subset, we have the corresponding full
subcategories $\sJ\subset \sI$, $\sJ^\circ\subset \sI^\circ$ and $\sJ^\circ_{\le
1}\subset \sI^\circ_{\le 1}$ and corresponding subdiagrams $\sD_x$. If the
collection of maps $x_i$ is understood, we write simply $\sD$ for $\sD_x$.

Let $F:A\to \sC$ be a functor, $a$ an object in $\sC$, $c_a:A\to \sC$ the constant
functor with value $a$  and $\phi:F\to c_a$ a natural transformation. Then $\phi$
induces a canonical map $\tilde\phi:\hocolim_AF\to a$ in $\sC$. As in the proof
of \cite[Proposition 4.3]{Spitzweck10}, let $C(A)$ be the category $A$ with a
final object $*$ adjoined and $C(F,\phi):C(A)\to \sC$ the functor with value $a$
on $*$, with restriction to $A$ being $F$ and which sends the unique map $y\to
*$ in $C(A)$, $y\in A$, to $\phi(y)$. Let  $[0,1]$ be the category with objects
$0, 1$ and a unique non-identity morphism, $0\to 1$ and let $C(A)^\Gamma$ be the
full subcategory of $C(A)\times[0,1]$ formed by removing the object $*\times 1$.
We extend $C(F,\phi)$ to a functor $C(F,\phi)^\Gamma:C(A)^\Gamma\to \sC$ by
$C(F,\phi)^\Gamma(y\times 1)=pt$, where $pt$ is the initial/final object in
$\sC$.

\begin{lemma}\label{lem:Cofib} There is a natural isomorphism in $\sC$
\[
\hocolim_{C(A)^\Gamma}C(F,\phi)^\Gamma\cong \hocofib(\tilde\phi:\hocolim_AF\to
a).
\]
\end{lemma}

\begin{proof} For a category $\sA$ we let $\sN(\sA)$ denote the simplicial nerve of $\sA$. We have an isomorphism of simplicial sets $\sN(C(A))\cong \text{Cone}(\sN(A), *)$, where $\text{Cone}(\sN(A), *)$ is the cone over $\sN(A)$ with vertex $*$.
Similarly, the full subcategory $A\times[0,1]$ of $C(A)^\Gamma$ has nerve
isomorphic to $\sN(A)\times\Delta[1]$. This gives an isomorphism of
$\sN(C(A)^\Gamma)$ with the push-out in the diagram
\[
\xymatrix{
\sN(A)\ar@{^(->}[r]\ar@{^(->}[d]_{\id\times\delta_0}&\text{Cone}(\sN(A), *)\\
\sN(A)\times\Delta[1]}
\]
This in turn gives an isomorphism of the simplicial object
$\hocolim_{C(A)^\Gamma}C(F,\phi)^\Gamma_*$  with the pushout in the diagram
\[
\xymatrix{
\hocolim_AF\ar@{^(->}[r]\ar@{^(->}[d]&C(\hocolim_AF,a)\\
C(\hocolim_AF, pt).}
\]
This gives the desired isomorphism.
\end{proof}

\begin{lemma}\label{lem:cofibration} Let $J\subset K\subset I$ be finite subsets
of $I$. Then the map 
\[
\hocolim_{\sJ_{\le1}^\circ}\sD_x\to \hocolim_{\sK_{\le 1}^\circ}\sD_x
\]
induced by the inclusion $J\subset K$ is a cofibration  in $\sC$.
\end{lemma}

\begin{proof} We give the category of simplicial objects in $\sC$,
$\sC^{\Delta^\op}$, the Reedy model structure, using the standard structure of a
Reedy category on $\Delta^\op$. By \cite[theorem 19.7.2(1), definition
19.8.1(1)]{Hirschhorn}, it suffices to show that 
\[
\hocolim_{\sJ_{\le1}^\circ}\sD_*\to \hocolim_{\sK_{\le 1}^\circ}\sD_*
\]
is a cofibration in $\sC^{\Delta^\op}$, that is, for each $n$, the map
\[
\phi_n:\hocolim_{\sJ_{\le1}^\circ}\sD_n\amalg_{L^n\hocolim_{\sJ_{\le1}^\circ}
\sD_*}L^n\hocolim_{\sK_{\le1}^\circ}\sD_*\to \hocolim_{\sK_{\le1}^\circ}\sD_n
\]
is a cofibration in $\sC$, where $L^n$ is the $n$th latching space. 

We note that 
\[
\hocolim_{\sJ_{\le1}^\circ}\sD_n=\bigvee_{\sigma\in
\sN(\sJ_{\le1}^\circ)_n}D(\sigma(0))
\]
where we view $\sigma\in \sN(\sJ_{\le1}^\circ)_n$ as a functor $\sigma:[n]\to
\sJ_{\le1}^\circ$; we have  a similar description of
$\hocolim_{\sK_{\le1}^\circ}\sD_n$. The latching space is
\[
L^n\hocolim_{\sJ_{\le1}^\circ}\sD_*=\bigvee_{\sigma\in
\sN(\sJ_{\le1}^\circ)_n^{deg}}D(\sigma(0)),
\]
where $\sN(\sJ_{\le1}^\circ)_n^{deg}$ is the subset of $\sN(\sJ_{\le1}^\circ)_n$
consisting of those $\sigma$ which contain an identity morphism;
$L^n\hocolim_{\sK_{\le1}^\circ}\sD_*$ has a similar description. The maps
\begin{gather*}
L^n\hocolim_{\sJ_{\le1}^\circ}\sD_*\to \hocolim_{\sJ_{\le1}^\circ}\sD_n,
L^n\hocolim_{\sJ_{\le1}^\circ}\sD_*\to L^n\hocolim_{\sK_{\le1}^\circ}\sD*, \\
L^n\hocolim_{\sK_{\le1}^\circ}\sD_*\to \hocolim_{\sK_{\le1}^\circ}\sD_n, 
\hocolim_{\sJ_{\le1}^\circ}\sD_n\to \hocolim_{\sK_{\le1}^\circ}\sD_n
\end{gather*}
 are
the  unions of identity maps on $D(\sigma(0))$ over the  respective inclusions 
 of the index sets. As $\sN(\sK_{\le1}^\circ)_n^{deg}\cap \sN(\sJ_{\le1}^\circ)_n=\sN(\sJ_{\le1}^\circ)_n^{deg}$, we have
\[
\hocolim_{\sJ_{\le1}^\circ}\sD_n\amalg_{L^n\hocolim_{\sJ_{\le1}^\circ}\sD_*}
L^n\hocolim_{\sK_{\le1}^\circ}\sD_*\cong \hocolim_{\sJ_{\le1}^\circ}\sD_n\bigvee
C,
\]
where 
\[
C=\bigvee_{\sigma\in \sN(\sK_{\le1}^\circ)_n^{deg}\setminus
\sN(\sJ_{\le1}^\circ)_n^{deg}}D(\sigma(0)),
\]
and the map to $ \hocolim_{\sK_{\le1}^\circ}\sD_n$ is the evident inclusion. As $D(N)$ is cofibrant for all $N$, 
this map is clearly a cofibration,  completing the proof.
\end{proof}

We have the $n$-cube $\sq^n$, the category associated to the partially ordered
set of subsets of $\{1,\ldots, n\}$, ordered under inclusion, and the punctured $n$-cube
$\sq^n_0$ of proper subsets.   We have the two inclusion functors $i_n^+,
i_n^-:\sq^{n-1}\to \sq^n$, $i_n^+(I):=I\cup\{n\}$, $i_n^-(I)=I$ and  the natural transformation $\psi_n:i_n^-\to i_n^+$ given as the collection of inclusions $I\subset
I\cup\{n\}$. The functor $i_n^-$ induces the functor
$i^-_{n0}:\sq^{n-1}\to \sq^n_0$.

For a functor $F:\sq^n\to \sC$, we have the iterated homotopy cofiber,
$\hocofib_nF$, defined inductively as the homotopy cofiber of
$\hocofib_{n-1}(F(\psi_n)):\hocofib(F\circ i_n^-)\to \hocofib(F\circ i_n^+)$.
Using this inductive construction, it is easy to define a natural isomorphism
$\hocofib_nF\cong \hocolim_{\sq^{n+1}_0}\hat{F}$, where $\hat{F}\circ
i^-_{n+10}=F$ and $\hat{F}(I)=pt$ if $n\in I$.

 The following result is proved in 
\cite[Lemma 4.2 and Proposition 4.3]{Spitzweck10}.

\begin{lemma}\label{lem:zeroth slice} Assume that $I$ is countable. Then
there is a canonical isomorphism in $\Ho\sC$
\[
\1/(\{x_i \ |\ i\in I\}) \cong \hocofib[\hocolim_{\sI^\circ} \sD_x \to
\hocolim_\sI\sD_x]
\]
\end{lemma}

\begin{proof}  As  $1$ is the final object in $\sI$, the collection of maps
$\times x^N: T_*^N\to\1$ defines a weak equivalence  $\pi:\hocolim_\sI\sD_x\to
\1$. In addition, for each $N\in \sI^\circ$, the comma category
$N/\sI^\circ_1$ has initial object the map $N\to \bar{N}$, where $\bar{N}_i=1$
if $N_i>0$, and $\bar{N}_i=0$ otherwise. Thus   $\sI^\circ_1$ is homotopy right
cofinal in  $\sI^\circ$  (see e.g. \cite[definition 19.6.1]{Hirschhorn}). Since
$\sD_x$ is a diagram of cofibrant objects in $\sC$, 
it follows from \cite[theorem 19.6.7]{Hirschhorn} that the map
$\hocolim_{\sI_1^\circ}\sD_x\to \hocolim_{\sI^\circ}\sD_x$
is a weak equivalence. This reduces us to identifying $\1/(\{x_i\})$ with the
homotopy cofiber of 
$\pi^\circ_{\le 1}:\hocolim_{\sI_1^\circ}\sD_x\to \1$, where $\pi^\circ_{\le 1}$
is the composition of $\pi$ with the natural map
$:\hocolim_{\sI_1^\circ}\sD_x\to \hocolim_{\sI}\sD_x$.

Next, we reduce to the case of a finite set $I$.   Take $I=\bN$. Let
$\sP_{fin}(I)$ be the category of finite subsets of $I$, ordered by inclusion,
 consider the full   subcategory $\sP^O_{fin}(I)$   of $\sP_{fin}(I)$
consisting of the subsets $I_n:=\{1,\ldots, n\}$ and let $\sI_{n,
\le1}^\circ\subset \sI_{ \le1}^\circ$ be the full subcategory with all indices
in $I_n$.  As $\sP^O_{fin}(I)$ is cofinal in $\sP_{fin}(I)$, we have
\[
\colim_n\hocolim_{\sI_{n, \le1}^\circ}\sD_x\cong \hocolim_{\sI_{\le 1}^\circ}
\sD_x.
\]

Take $n\le m$. By lemma~\ref{lem:cofibration} the the map 
$\hocolim_{\sI_{n, \le 1}^\circ}\sD_x\to \hocolim_{\sI_{m, \le 1}^\circ}\sD_x$
is a cofibration  in $\sC$. Thus, using the Reedy model structure on $\sC^\bN$
with $\bN$ considered as a direct category, the $\bN$-diagram in $\sC$,
$n\mapsto \hocolim_{\sI_{n, \le1}^\circ}\sD_x$, is a cofibrant object in
$\sC^\bN$. As $\bN$ is a direct category,  the fibrations in  $\sC^\bN$ are the
pointwise ones, hence $\bN$ has pointwise constants \cite[definition
15.10.1]{Hirschhorn} and therefore \cite[theorem 19.9.1]{Hirschhorn} the
canonical map
\[
\hocolim_{n\in\bN}\hocolim_{\sI_{n, \le 1}^\circ}\sD_x\to
\colim_{n\in\bN}\hocolim_{\sI_{n, \le 1}^\circ}\sD_x
\]
is a weak equivalence in $\sC$. This gives us the weak equivalence in $\sC$
\[
\hocolim_n\hocolim_{\sI_{n,\le 1}^\circ}\sD_x\to \hocolim_{\sI_{\le
1}^\circ}\sD_x.
\]
Since $\bN$ is contractible, the canonical map $\hocolim_\bN\1\to \1$ is a weak
equivalence in $\sC$, giving us the weak equivalences
\begin{multline*}
\hocofib[ \hocolim_{\sI_{\le 1}^\circ}\sD_x\to\1]\\\sim
\hocofib[\hocolim_{n\in\bN}\hocolim_{\sI_{n, \le 1}^\circ}\sD_x\to
\hocolim_{n\in\bN}\1]\\\sim
\hocolim_{n\in\bN}[\hocofib[\hocolim_{\sI_{n, \le 1}^\circ}\sD_x\to\1]].
\end{multline*}
Thus, we need only exhibit isomorphisms in $\Ho\sC$
\[
\rho_n:\hocofib[\hocolim_{\sI_{n,\le 1}^\circ}\sD_x\to\1]\to  \1/(x_1,\ldots,
x_n):=\1/(x_1)\otimes\ldots\otimes\1/(x_n),
\]
which are natural in $n\in\bN$.

By lemma~\ref{lem:Cofib} we have a natural isomorphism in $\sC$, 
\[
\hocofib[\hocolim_{\sI_{n,\le 1}^\circ}\sD_x\to\1]\cong
\hocolim_{C(\sI_{n,\le 1}^\circ)^\Gamma}C(\sD_x, \pi)^\Gamma. 
\]
However, $\sI_{n,\le 1}^\circ$ is isomorphic to $\sq^n_0$ by sending
$N=(N_1,\ldots, N_n)$ to $I(N):=\{i \ |\ N_i=0\}$. Similarly,  $C(\sI_{n,\le
1}^\circ)$ is isomorphic to $\sq^n$, and $C(\sI_{n,\le 1}^\circ)^\Gamma$ is thus
isomorphic to $\sq^{n+1}_0$. From our discussion above, we see that 
$\hocolim_{C(\sI_{n,\le 1}^\circ)^\Gamma}C(\sD_x, \pi)^\Gamma$ is isomorphic to 
$\hocofib_nC(\sD_x, \pi)$, so we need only exhibit  isomorphisms in $\Ho\sC$
\[
\rho_n:\hocofib_nC(\sD_x, \pi)\to  \1/(x_1)\otimes\ldots\otimes\1/(x_n)
\]
which are natural in $n\in\bN$.

We do this inductively as follows. To include the index $n$ in the notation, we
write $C(\sD_x, \pi)_n$ for the functor $C(\sD_x, \pi):\sq^n\to \sC$. For $n=1$,
$\hocofib_1C(\sD_x, \pi)_1$ is the mapping cone of $\mu_1\circ (\times x_1\otimes\id):\1\otimes T_1\otimes\1\to \1$, which is isomorphic in $\Ho\sC$ to the homotopy cofiber of $\times
x_1:\1\otimes T_1\to \1$. As this latter is equal to $\1/(x_1)$, so we take
$\rho_1:\hocofib_1C(\sD_x, \pi)_1\to \1/(x_1)$ to be this isomorphism. We note that
$C(\sD_x, \pi)_n\circ i_n^+=C(\sD_x, \pi)_{n-1}$ and  
$C(\sD_x, \pi)_n\circ i_n^-=C(\sD_x, \pi)_{n-1}\otimes T_n\otimes\1$.

Define  $C(\sD_x, \pi)_n'$ by $C(\sD_x, \pi)'_n\circ i_n^-=C(\sD_x, \pi)_{n-1}\otimes\1\otimes T_n\otimes\1$,
$C(\sD_x, \pi)'_n\circ i_n^+=C(\sD_x, \pi)_{n-1}\otimes \1$, with the natural transformation $C(\sD_x,
\pi)_n'\circ\psi_n$  given as
\[
C(\sD_x, \pi)_{n-1}\otimes\1\otimes T_n\otimes\1\xrightarrow{(\id\otimes\mu)\circ(\id\otimes \times
x_n\otimes\id_\1)}C(\sD_x, \pi)_{n-1}\otimes\1.
\]
The evident multiplication maps give a weak equivalence 
$C(\sD_x, \pi)_n'\to C(\sD_x, \pi)_n$, giving us the isomorphism in $\Ho\sC$
\[
\rho_n:\hocofib_nC(\sD_x, \pi)_n\to\1/(x_1)\otimes\ldots \otimes  \1/(x_n) 
\]
defined as the composition
\begin{align*}
\hocofib_nC&(\sD_x, \pi)_n\cong \hocofib_nC(\sD_x, \pi)_n' \\
&\cong \hocofib(\hocofib_{n-1}(C(\sD_x, \pi)_{n-1}\otimes\1\otimes T_n)\\
&\hskip30pt\xrightarrow{\hocofib_{n-1}(\id\otimes \times x_n)}
\hocofib_{n-1}(C(\sD_x, \pi)_{n-1}\otimes \1))\\
&\cong  \hocofib(\hocofib_{n-1}(C(\sD_x, \pi)_{n-1})\otimes\1\otimes  T_n\\
&\hskip30pt\xrightarrow{\id\otimes\times x_n}
\hocofib_{n-1}(C(\sD_x, \pi)_{n-1})\otimes \1)\\
&\cong  \hocofib_{n-1}(C(\sD_x, \pi)_{n-1})\otimes \hocofib( \times
x_n:\1\otimes T_n\to  \1)\\
&=\hocofib_{n-1}(C(\sD_x, \pi)_{n-1})\otimes \1/(x_n)\\
&\xrightarrow{\rho_{n-1}\otimes\id}\1/(x_1)\otimes\ldots\otimes\1/(x_{n-1}
)\otimes  \1/(x_n).
\end{align*}

Via the definition of $\hocofib_n$, 
\begin{multline*}
\hocofib_nC(\sD_x, \pi)_n=\hocofib[\hocofib_{n-1}(C(\sD_x, \pi)_n\circ
i_n^-)\\\xrightarrow{\hocofib_{n-1}(C(\sD_x, \pi)_{n-1}(\psi_n))} 
\hocofib_{n}(C(\sD_x, \pi)_n\circ i_n^+]
\end{multline*}
and the identification $C(\sD_x, \pi)_n\circ i_n^+=C(\sD_x, \pi)_{n-1}$, we have
the canonical map $\hocofib_{n-1}(C(\sD_x, \pi)_{n-1})\to \hocofib_n(C(\sD_x,
\pi)_n$. One easily sees that the diagram
\[
\xymatrixcolsep{50pt}
\xymatrix{
\hocofib[\hocolim_{\sI_{n-1,\le
1}^\circ}\sD_x\to\1]\ar[d]_\sim\ar[r]&\hocofib[\hocolim_{\sI_{n,\le
1}^\circ}\sD_x\to\1]\ar[d]^\sim\\
\hocofib_{n-1}(C(\sD_x,
\pi)_{n-1})\ar[r]\ar[d]_-{\rho_{n-1}}&\hocofib_n(C(\sD_x,
\pi)_n)\ar[d]^-{\rho_n}\\
\1/(x_1)\otimes\ldots\otimes\1/(x_{n-1})\ar[r]_-{\rho_{\{1,\ldots,n-1\}\subset
\{1,\ldots, n\}}}
&\1/(x_1)\otimes\ldots\otimes  \1/(x_n)
}
\]
commutes in $\Ho\sC$, giving the desired naturality in $n$.
\end{proof}

Now let $M$ be an object in $\sC$, let $QM\to M$ be a cofibrant replacement and
form the $\sI$-diagram $\sD_x\otimes QM:\sI\to \sC$, $(\sD_x\otimes
QM)(N)=\sD_x(N)\otimes QM$. 

\begin{proposition}\label{prop:zeroth slice}  Assume that  $I$ is countable. Let
$M$ be an object in $\sC$. Then there is a canonical isomorphism in $\Ho\sC$
\[
M/(\{x_i \ |\ i\in I\}) \cong \hocofib[\hocolim_{\sI^\circ} \sD_x\otimes QM \to
\hocolim_\sI\sD_x\otimes QM]
\]
\end{proposition}

\begin{proof} This follows directly from lemma~\ref{lem:zeroth slice}, noting
the definition of $M/(\{x_i \ |\ i\in I\})$ as $[\1/(\{x_i \ |\ i\in
I\})]\otimes QM$ and the canonical isomorphism
\begin{multline*}
\hocofib[\hocolim_{\sI^\circ} \sD_x\otimes QM \to \hocolim_\sI\sD_x\otimes
QM]\\\cong
\hocofib[\hocolim_{\sI^\circ} \sD_x \to \hocolim_\sI\sD_x]\otimes QM
\end{multline*}
\end{proof}

\begin{proposition}
\label{prop:induction step for higher slices}
Let $\sF:\sI_{\deg  \geq n}\to \sC$ 
be a diagram in a cofibrantly generated model category
$\sC$. Suppose for every monomial $M$ of degree $n$ the natural map
$\hocolim \sF|_{\sI_{> M}}\to \sF(M)$
is a weak equivalence. Then the natural map 
\[
\hocolim \sF|_{\deg \geq n+1}
\to \hocolim \sF
\]
is a weak equivalence.
\end{proposition}

\begin{proof} 
This is just \cite[lemma 4.4]{Spitzweck10}, with the following corrections: the
statement of the lemma in {\it loc.  cit.} has 
``$\hocolim\sF|_{\sI_{\ge M}}\to \sF(M)$ is a weak equivalence'' rather than the
correct assumption
``$\hocolim\sF|_{\sI_{> M}}\to \sF(M)$ is a weak equivalence'' and in the proof,
one should replace the object $Q(M)$ with $\mathrm{colim\,} Q|_{I>M}$ rather
than with $\mathrm{colim\,} Q|_{I\ge M}$. \end{proof}

\section{Slices of effective motivic module
spectra}\label{sec:SlicesModuleSpectra}

In this section we will describe the slices for  modules for a 
commutative and effective ring $T$-spectrum $\sR$ that satisfies certain additional conditions. We adapt  the constructions used in describing slices of 
$\MGL$ in \cite{Spitzweck10}, which go through without significant change in
this more general  setting. 

Let us first recall the definition of the slice tower in  $\SH(S)$. We will use
the standard model category $\Mot:=\Mot(S)$ of  symmetric $T$-spectra over $S$, 
$T:=\A^1/\A^1\setminus\{0\}$, with the motivic model structure as in
\cite{Jardine}, for defining the triangulated tensor category
$\SH(S):=\Ho\Mot(S)$.

For an integer $q$, let $\Sigma_T^q\SH^\eff(S)$ denote the localizing
subcategory of $\SH(S)$ generated by $\sS_q:=\{\Sigma_T^q\Sigma_T^\infty X_+\ |\
p\ge q, X\in \Sm/S\}$, that is $\Sigma_T^q\SH^\eff(S)$ is the smallest
triangulated  subcategory of $\SH(S)$ which contains  $\sS_q$ and is closed
under direct sums and isomorphisms in $\SH(S)$. This gives a filtration on 
$\SH(S)$ by full localizing subcategories
\[
\cdots  \subset \Sigma^{q+1}_T \SH^\eff(S) \subset
\Sigma^{q}_T \SH^\eff(S)   \subset
\Sigma^{q-1}_T  \SH^\eff(S) \subset   \cdots\subset \SH(S).
\]

The set $\sS_q$  is a set of compact generators of $ \Sigma^{q}_T \SH(S)$ and the set $\cup_q\sS_q$ is similarly a set of compact generators for $\SH(S)$. By Neeman's triangulated version of Brown representability theorem \cite{Neeman}, the 
inclusion $i_q: \Sigma^{q}_T \SH^\eff(S)  
\to \SH(S)$ has a right adjoint $r_q: \SH(S)  \to \Sigma^{q}_T
\SH^\eff(S)$. We let $f_q:=i_q\circ r_q$. The inclusion  $ \Sigma^{q+1}_T
\SH^\eff(S)\to  \Sigma^{q}_T
\SH^\eff(S)$ induces a canonical natural transformation $f_{q+1}   \to f_{q}$. 
Putting these together forms the slice tower
\begin{equation}
\label{slice tower}
\cdots\to f_{q+1}\to f_q\to\cdots\to \id.
\end{equation}
For each $q$ there exist triangulated functor 
$s_q:\SH(S) \to \SH(S)$ 
such that for every $\sE\in \SH(S),\ s_q (\sE) \in
\Sigma^{q}_T \SH^\eff(S)$ and there is a  canonical and natural distinguished
triangle 
\[
f_{q+1}(\sE)  \to f_q(\sE) \to s_q(\sE)  \to 
\Sigma f_{q+1}(\sE)
\]
in $\SH(S)$.

Pelaez has given a lifting of the construction of the functors $f_q$ to the
model category level. For this, he starts with the model category $\Mot$ and
forms for each $n$ the right Bousfield localization of $\Mot$ with respect to
the objects $\Sigma^m_T  F_n X_+$ with $m-n\ge q$ and $X\in \Sm/S$. Here $F_n
X_+$ is the shifted $T$-suspension spectrum, that is, $\Sigma_T^{m-n}X_+$ in
degree $m\ge n$, $pt$ in degree $m<n$, and with identity bonding maps. Calling
this Bousfield localization $\Mot_q$, the functor $r_q$ is given by taking a functorial
cofibrant replacement in $\Mot_q$. As the underlying categories are all the
same, this gives liftings $\tilde f_q$ of $f_q$ to endofunctors on $\Mot$. The
technical condition on $\Mot$ invoked by Pelaez is that of {\em cellularity} and
right properness, which ensures that the right Bousfield localization exists;
this follows from the work of Hirschhorn \cite{Hirschhorn}. Alternatively, one can use the fact
that $\Mot$ is a {\em combinatorial} right proper model category, following
work 
of J. Smith, detailed for example in \cite{Barwick}. 

The combinatorial property passes to module categories, and so this approach
will be useful here. The category $\Mot$ is a closed symmetric monoidal
simplicial model category, with cofibrant unit the sphere (symmetric) spectrum
$\mS_S$ and product $\wedge$. Let  $\sR$ be a commutative monoid in $\Mot$. We
have the model category $\sC:=\sR$-$\Mod$ of $\sR$-modules, as constructed in
\cite{SchwedeShipley}.  The fibrations and weak equivalences are the morphisms
which are  fibrations, resp. weak equivalences, after applying the forgetful
functor to $\Mot$; cofibrations are those maps having the left lifting property
with respect to trivial fibrations. This makes  $\sC$ into a pointed closed symmetric
monoidal simplicial model category; $\sC$ is in addition cofibrantly generated
and combinatorial. Assuming that $\sR$ is a cofibrant object in $\Mot$, the free
$\sR$-module functor, $\sE\mapsto \sR\wedge\sE$, gives a left adjoint to the
forgetful functor and gives rise to a Quillen adjunction.  For details as to
these  facts and a general construction of this model category structure on module
categories, we refer the reader to \cite{SchwedeShipley}; another source is 
\cite{Hovey}, especially theorem 1.3, proposition 1.9 and proposition 1.10.

The model category $\sR$-$\Mod$ inherits right properness from $\Mot$. We may therefore form the right Bousfield localization $\sC_q$ with respect to
the free $\sR$-modules $\sR\wedge \Sigma^m_T  F_n X_+$ with $m-n\ge q$ and $X\in
\Sm/S$, and define the endofunctor $\tilde{f}_q^\sR$ on $\sC$ by taking a
functorial cofibrant replacement in $\sC_q$. By the adjunction, one sees that
$\Ho\sC_q$  is equivalent to the localizing subcategory of $\Ho\sC$ (compactly)
generated by $\{\sR\wedge\Sigma^m_T  F_n X_+\ |\ m-n\ge q, X\in \Sm/S\}$. We
denote this localizing subcategory by $\Sigma^q_T\Ho\sC^\eff$, or $\Ho\sC^\eff$
for $q=0$. We call an object $\sM$ of $\sC$ {\em effective} if the image of
$\sM$ in $\Ho\sC$ is in $\Ho\sC^\eff$, and denote the full subcategory of effective objects of $\sC$ by $\sC^\eff$.

 Just as above, Neeman's results give a right adjoint $r^\sR_q$ to the inclusion
$i^\sR_q:\Sigma^q_T\sC^\eff\to \sC$ and the composition $f_q^\sR:=i_q^\sR\circ
r_q^\sR$ is represented by $\tilde{f}_q^\sR$. One recovers the functors $f_q$
and $\tilde{f}_q$ by taking $\sR=\mS_S$.

\begin{lemma}
\label{lem:properties of slice functor} Let $\sR$ be a cofibrant commutative
monoid in $\Mot$.
The functors $f^\sR_q:\Ho\sC\to \Ho\sC$ and their liftings $\tilde{f}^\sR_q$ 
have the following properties.
\begin{enumerate}
 \item Each $f^\sR_n$ is idempotent, i.e., $(f^\sR_n)^2=f^\sR_n$.  
 \item $f^\sR_n \Sigma^1_T= \Sigma^1_T f^\sR_{n-1}$ for $n\in\Z$.
  \item Each $\tilde{f}^\sR_n$ commutes with homotopy colimits.
\item Suppose that $\sR$ is in $\SH^\eff(S)$. Then the forgetful functor
$U:\Ho\sR$-$\Mod \to\SH(S)$ induces an isomorphism $U\circ f_q^\sR\cong f_q\circ
U$ as well as an isomorphism $U\circ s_q^\sR\cong s_q\circ U$, for all $q\in\Z$.
\end{enumerate}
\end{lemma}

\begin{proof}
(1) and (2) follow from universal property of triangulated functors $f^\sR_n$.
In case $\sR=\mS_S$,  (3) is 
proved in \cite[Cor 4.5]{Spitzweck10}; the proof for general $\sR$ is the same.
For (4), it suffices to prove the result for $f_q$ and $f^\sR_q$.  Take $\sM\in
\sC$. We check the universal property of $Uf^\sR_q\sM\to \sM$: Since $\sR$ is in
$\SH^\eff(S)$ and the functor $-\wedge\sR$ is compatible with homotopy cofiber
sequences and direct sums,  $-\wedge\sR$ maps $\Sigma^q_T\SH^\eff(S)$ into
itself for each $q\in \Z$. As $U(\sR\wedge \sE)=\sR\wedge \sE$, it follows that
$U(\Sigma^q_T\Ho\sR$-$\Mod^\eff)\subset \Sigma^q_T\SH^\eff(S)$ for each $q$. In
particular, $U(f^\sR_q(\sM))$ is in $\Sigma^q_T\SH^\eff(S)$.  For $p\ge q$,
$X\in \Sm/S$,  we have
\begin{align*}
\Hom_{\SH(S)}(\Sigma^p_T\Sigma_T^\infty X_+, U(f^\sR_q(\sM)))&\cong 
\Hom_{\Ho\sC}(\sR\wedge\Sigma^p_T\Sigma_T^\infty X_+, f^\sR_q(\sM))\\
&\cong \Hom_{\Ho\sC}(\sR\wedge\Sigma^p_T\Sigma_T^\infty X_+, \sM)\\
&\cong \Hom_{\SH(S)}(\Sigma^p_T\Sigma_T^\infty X_+, U(\sM))
\end{align*}
so the canonical map $U(f^\sR_q(\sM))\to f_q(U(\sM))$ is therefore an
isomorphism.
\end{proof}

From the adjunction $\Hom_\sC(\sR, \sM)\cong \Hom_{\Mot}(\mS_S, \sM)$
and the fact that $\mS_S$ is a cofibrant object of $\Mot$, we see that $\sR$ is
a cofibrant object of $\sC$. Thus $\sC$ is a closed symmetric monoidal
simplicial model category with cofibrant unit $1:=\sR$ and monoidal product
$\otimes=\wedge_\sR$.  Similarly, $T_\sR:=\sR\wedge T$ is a cofibrant object of
$\sC$. Abusing notation, we write $\Sigma_T(-)$ for the endofunctor $A\mapsto
A\otimes T_\sR$ of $\sC$.

We recall that the category $\Mot$ satisfies the monoid axiom of Schwede-Shipley
\cite[definition 3.3]{SchwedeShipley}; the reader can see for example    the
proof of \cite[lemma 4.2]{Hoyois}. Following remark~\ref{rem:Fibrant}, there is
a fibrant replacement $\sR\to \1$ in $\sC$ such that $\1$ is an $\sR$-algebra;
in particular, $\sR\to \1$ is a cofibration and a weak equivalence in both $\sC$
and in $\Mot$, and $\1$ is fibrant in in both $\sC$ and in $\Mot$.

For each $\bar{x}\in \sR^{-2d,-d}(S)$, we have the corresponding element
$\bar{x}:T_\sR^{\otimes d}\to \sR$ in $\Ho\sC$, which we may lift to a morphism
$x:T_\sR^{\otimes d}\to \1$ in $\sC$. Thus, for a collection of elements 
$\{\bar{x}_i\in \sR^{-2d_i,-d_i}(S)\ |\ i\in I\}$, we have the associated
collection of maps in $\sC$, $\{x_i:T_\sR^{\otimes d_i}\to \1\ |\ i\in I\}$ and
thereby the quotient object $\1/(\{x_i\})$ in $\sC$. Similarly, for $\sM$ an
$\sR$-module, we have the $\sR$-module $\sM/(\{x_i\})$, which is a cofibrant
object in $\sC$. We often write $\sR/(\{x_i\})$ for $\1/(\{x_i\})$.

\begin{lemma}
\label{effectivity of quotients of effective spectra}
Suppose that $\sR$ is in $\SH^\eff(S)$.
Then for any set 
\[
\{\bar{x}_i\in \sR^{-2d_i,-d_i}(S)\ |\ i\in I, d_i>0\} 
\]
of elements of $\sR$-cohomology,  the  object $\sR/(\{x_i\})$ is effective. If
in addition $\sM$ is an $\sR$-module and is effective, then $\sM/(\{x_i\})$ is
effective.
\end{lemma}

\begin{proof}
This follows from   lemma~\ref{lem:properties of slice functor} since $f^\sR_n$
is a triangulated functor and $\sC^\eff$ is closed under homotopy colimits.
\end{proof}

Let $A$ be an abelian group and $SA$ the topological sphere spectrum with 
$A$-coefficients. For a $T$-spectrum $\sE$ 
let us denote the spectrum $\sE\wedge (SA)$ 
by $\sE\otimes A$. Of course, if $A$ is the free abelian group on a set $S$,
then $\sE\otimes A=\oplus_{s\in S}\sE$.

Let $\{\bar{x}_i\in \sR^{-2d_i,-d_i}(S)\ |\ i\in I, d_i>0\} $
be a set of elements of $\sR$-cohomology,  with $I$ countable.  Suppose that
$\sR$ is cofibrant as an object in $\Mot$ and is in $\SH^\eff(S)$. Let $\sM$ be
in $\sC^\eff$ and let $Q\sM\to \sM$ be a cofibrant replacement. By  lemma
\ref{lem:zeroth slice}, we have a homotopy cofiber sequence in $\sC$, 
\[
\hocolim_{\sI^\circ}\sD_x\otimes Q\sM\to Q\sM\to \sM/(\{x_i\}).
\]
Clearly $\hocolim_{\sI^\circ}\sD_x\otimes Q\sM$ is in $\Sigma^1_T\Ho\sC^\eff$,
hence the above sequence induces an isomorphism in $\Ho\sC$
\[
s^\sR_0\sM\xrightarrow{\sigma_\sM}s^\sR_0( \sM/(\{x_i\})).
\]
Composing the canonical map $ \sM/(\{x_i\})\to s^\sR_0( \sM/(\{x_i\}))$ with
$\sigma_\sM^{-1}$ gives the canonical map
\[
\pi^\sR_\sM: \sM/(\{x_i\})\to s^\sR_0\sM
\]
in $\Ho\sC$.  Applying the forgetful functor gives the canonical  map in $\SH(S)$
\[
\pi_\sM: U(\sM/(\{x_i\}))\to U(s^\sR_0\sM)\cong s_0(U\sM).
\]
This equal to  the canonical map
$U(\sM/(\{x_i\}))\to s_0(U(\sM/(\{x_i\})))$ composed with the inverse of the isomorphism
$s_0(U\sM)\to s_0(U(\sM/(\{x_i\})))$.

\begin{theorem}\label{thm:slices of effective spectra}
Let $\sR$ be a commutative monoid in $\Mot(S)$, cofibrant as an object in
$\Mot(S)$, such that $\sR$ is in $\SH^\eff(S)$.
Let $X=\{\bar{x}_i\in \sR^{-2d_i,-d_i}(S)\ |\ i\in I, d_i>0\} $
be a countable set of elements of $\sR$-cohomology. Let $\sM$ be an
$\sR$-module in $\sC^\eff$ and suppose that the  canonical map $\pi_\sM:
U(\sM/(\{x_i\}))\to  s_0(U\sM)$ is an isomorphism. Then for each $n\ge0$, we
have a canonical isomorphism in $\Ho\sC$,
\[
s^\sR_n\sM\cong \Sigma^{n}_T s^\sR_0\sM
\otimes \Z[X]_n,
\]
where $\Z[X]_n$ is the abelian group of weighted-homogeneous degree $n$ polynomials 
over $\Z$ in the variables $\{x_i, i\in I \}$, $\deg x_i=d_i$. Moreover, for each $n$, we have a
canonical isomorphism in $\SH(S)$,
\[
s_nU\sM\cong \Sigma^{n}_T s_0U\sM
\otimes \Z[X]_n.
\]
\end{theorem}

\begin{proof} Replacing $\sM$ with a cofibrant model, we may assume that $\sM$
is cofibrant in $\sC$; as $\sR$ is cofibrant in $\Mot$, it follows that $U\sM$
is cofibrant in $\Mot$.

Since $\pi_\sM=U(\pi^\sR_\sM)$, our assumption on $\pi_\sM$ is the same as assuming that
$\pi^\sR_\sM$ is an isomorphism in $\Ho\sC$. By construction, $\pi^\sR_\sM$
extends to a map of distinguished triangles
\[
\xymatrix{
(\hocolim_{\sI^\circ}\sD_x)\otimes \sM\ar[r]\ar[d]_\alpha&\sM\ar[r]\ar@{=}[d]&\sM/(\{x_i\})\ar[d]^{\pi^\sR_\sM}\ar[r]&\Sigma(\hocolim_{\sI^\circ}\sD_x)\otimes \sM\ar[d]^{\Sigma\alpha}\\
f^\sR_1\sM\ar[r]&\sM\ar[r]& s^\sR_0\sM\ar[r]&\Sigma f^\sR_1\sM,
}
\]
and thus the map $\alpha$ is an isomorphism. We note that $\alpha$ is equal to
the canonical map given by the universal property of $f^\sR_1\sM\to\sM$.

We will now identify $f^\sR_n\sM$   in terms of the
diagram $\sD_x|_{\sI_{\deg \ge n}}\otimes \sM$, proving by induction on $n\ge1$
that
the canonical map $\hocolim \sD_x\otimes \sM|_{\deg \ge n}\to f^\sR_n\sM$ in
$\Ho\sC$ is an isomorphism.

As $\sI^\circ=\sI_{\deg\ge1}$, the case $n=1$ is settled.  Assume the result for
$n$. We claim that the diagram
\[
\tilde{f}^\sR_{n+1}[\sD_x\otimes\sM|_{\deg \ge n}]:\sI_{\deg \ge n}\to \sC
\]
satisfies the hypotheses of  proposition~\ref{prop:induction step for higher
slices}. 
That is,  we need to verify that for every monomial $M$ of degree $n$
the natural map 
\[
\hocolim  \tilde{f}^\sR_{n+1}[\sD_{> M}\otimes\sM]\to 
\tilde{f}^\sR_{n+1}[\sD(M)\otimes\sM]
\]
is a weak equivalence in 
$\sC$.  This follows by the string of isomorphisms in $\Ho\sC$
\begin{align*}
\hocolim \tilde{f}^\sR_{n+1} \sD_{> M}\otimes\sM&\cong\hocolim
\tilde{f}^\sR_{n+1}
\Sigma^n_T\sD_{\deg  \geq 1}\otimes\sM \\
&\cong \hocolim  \Sigma^n_T \tilde{f}^\sR_{1}
\sD_{\deg  \geq 1}\otimes\sM\\
&\cong\Sigma^n_T   f^\sR_{1}\hocolim\sD_{\deg \geq 1}\otimes\sM\\
&\cong \Sigma^n_T f^\sR_{1}  f^\sR_1 \sM\\
&\cong \Sigma^n_T f^\sR_{1} \sM\\
&\cong f^\sR_{n+1}  \Sigma^n_T \sM\\
&\cong  f^\sR_{n+1}[\sD(M)\otimes\sM].
\end{align*}
Applying proposition~\ref{prop:induction step for higher slices} and our
induction hypothesis gives us the string of  isomorphisms in $\Ho\sC$
\begin{multline*}
f^\sR_{n+1}\sM\cong  f^\sR_{n+1}f^\sR_n\sM\cong
f^\sR_{n+1}\hocolim[\sD_x\otimes\sM|_{\deg \ge
n}]\\
\cong
 \hocolim \tilde{f}^\sR_{n+1}[\sD_x\otimes\sM|_{\deg \ge
n}]
\cong \hocolim \tilde{f}^\sR_{n+1}[\sD_x\otimes\sM|_{\deg \ge n+1}]\\
\cong  \hocolim\sD_x\otimes\sM|_{\deg \ge n+1},
\end{multline*}
the last isomorphism following from the fact that $\sD_x(x^N)\otimes\sM$ is in
$\Sigma^{|N|}_T\sC^\eff$, and hence the canonical map 
$\tilde{f}^\sR_{n+1}[\sD_x\otimes\sM]\to  \sD_x\otimes\sM$ is an objectwise weak equivalence
on $\sI_{\deg\ge n+1}$.

For the  slices $s_n$ we have
\begin{align*}
s^\sR_n\sM&:= \hocofib(\tilde{f}^\sR_{n+1}\sM\to \tilde{f}^\sR_n\sM)\cong
 \hocofib(\tilde{f}^\sR_{n+1}\tilde{f}^\sR_n\sM\to \tilde{f}^\sR_n\sM)\\
&\cong \hocofib( \hocolim  \tilde{f}^\sR_{n+1}[\sD_{\deg \geq n}\otimes\sM]\to
\hocolim\sD_{\deg  \geq n}\otimes\sM)\\ 
&\cong \hocolim\hocofib(\tilde{f}^\sR_{n+1}[\sD_{\deg \geq n}\otimes\sM]\to  
\sD_{\deg  \geq n}\otimes\sM).
\end{align*}
At a monomial of degree greater than 
$n$,  the canonical map $\tilde{f}^\sR_{n+1}[\sD_{\deg\geq n}\otimes\sM]\to  
\sD_{\deg  \geq n}\otimes\sM$ is a weak equivalence, and at a monomial $M$ of
degree $n$ the homotopy cofiber is given by 
\begin{multline*}
\hocofib(\tilde{f}^\sR_{n+1}[\sD(M)\otimes\sM]\to  \sD(M)\otimes\sM)
=\hocofib((\tilde{f}^\sR_{n+1}[\Sigma^n_T\sM]\to  \Sigma^n_T\sM)\\
\cong \hocofib(\Sigma^n_T \tilde{f}^\sR_1\sM\to   \Sigma^n_T\sM)
\cong \Sigma^n_T  s^\sR_0\sM
\end{multline*}

Let $\tilde{s}_0^\sR$ be the functor on $\sC^\eff$, $\sN\mapsto \hocofib(\tilde{f}^\sR_1\sN\to \sN)$,
and let $F_n\sM:\sI_{\deg \ge n}\to \sC^\eff$ be the diagram 
\[
F_n(M)=\begin{cases} 
pt&\text{ for }\deg M>n\\ \Sigma^n_T\tilde{s}^\sR_0\sM&\text{  for } \deg
M=n. 
\end{cases}
\]
We thus have a weak equivalence of pointwise cofibrant functors
\[
\hocofib(\tilde{f}^\sR_{n+1}[\sD_{\deg \geq n}\otimes\sM]\to  
\sD_{\deg  \geq n}\otimes\sM)\to F_n:\sI_{\deg\ge n}\to \sC,
\]
and therefore a weak equivalence on the homotopy colimits. As we have the
evident isomorphism in $\Ho\sC$
\[
\hocolim_{\sI_{\deg\ge n}}F_n\cong \oplus_{M, \deg M=n}\Sigma^{n}_T
s^\sR_0\sM,
\]
 this gives us the desired isomorphism $s^\sR_n\sM\cong  \Sigma^{n}_T
s^\sR_0\sM\otimes \Z[X]_n$ in $\Ho\sC$. Applying the forgetful functor and
using lemma~\ref{lem:properties of slice functor}  gives the isomorphism
 $s_nU\sM\cong  \Sigma^{n}_T s_0U\sM\otimes\Z[X]_n$  in $\SH(S)$.
\end{proof}

\begin{corollary}\label{cor:slices of localized spectra}
Let $\sR$, $X$ and $\sM$ be as in theorem~\ref{thm:slices of effective spectra}.
Let $Z=\{z_j\in \Z[X]_{e_j}\}$ be a collection of homogeneous
 elements of $\Z[X]$,  and let $\sM[Z^{-1}]\in \sC$ be the localization
of $\sM$ with respect to the collection of maps $\times z_j:\sM\to
\Sigma^{-e_j}_T\sM$. Then there are natural isomorphisms
\[
s^\sR_n\sM[Z^{-1}]\cong \Sigma^{n}_T s^\sR_0\sM
\otimes \Z[X][Z^{-1}]_n,
\]
\[
s_nU\sM[Z^{-1}]\cong \Sigma^{n}_T s_0U\sM
\otimes \Z[X][Z^{-1}]_n.
\]
\end{corollary}

\begin{proof} Each map $\times z_j:\sM\to \Sigma^{-e_j}_T\sM$ induces the
isomorphism  $\times z_j:\sM[Z^{-1}]\to \Sigma^{-e_j}_T\sM[Z^{-1}]$ in $\Ho\sC$, with inverse
$\times z^{-1}_j:\Sigma_T^{-e_j}\sM[Z^{-1}]\to \sM[Z^{-1}]$. Applying $f_q^\sR$ gives us the map in $\Ho\sC$
\[
\times z_j:f_q^\sR\sM\to f_q^\sR\Sigma^{-e_j}_T\sM\cong
\Sigma^{-e_j}_Tf_{q+e_j}^\sR\sM.
\]

As $f_{q+e_j}^\sR\sM$ is in $\Sigma_T^{q+e_j}\Ho\sC^\eff$,  both
$\Sigma^{-e_j}_Tf_{q+e_j}^\sR\sM
$ and $f_q^\sR\sM$ are in $\Sigma_T^q\Ho\sC^\eff$. The composition
\[
\Sigma^{-e_j}_Tf_{q+e_j}^\sR\sM\to \Sigma^{-e_j}_T \sM\xrightarrow{\times
z^{-1}_j}\sM[Z^{-1}]
\]
gives via the universal property of $f_q^\sR$ the map $\Sigma^{-e_j}_Tf_{q+e_j}^\sR\sM\to
f_q^\sR\sM[Z^{-1}]$. Setting $|N|=\sum_jN_je_j$, this extends to give a map of the system of monomial
multiplications
\[
\times z^{N-M}:\Sigma^{-|N|}_Tf_{q+|N|}^\sR\sM\to
\Sigma^{-|M|}_Tf_{q+|M|}^\sR\sM
\]
to $f_q^\sR\sM[Z^{-1}]$; the universal property of the truncation functors $f_n$
and of localization shows that this system induces an isomorphism
\[
\hocolim_{N\in \sI^\op}\Sigma^{-|N|}_Tf_{q+|N|}^\sR\sM\cong f^\sR_q\sM[Z^{-1}]
\]
in $\Ho\sC$.  As the slice functors $s_q$ are exact and commute with $\hocolim$, we have a similar collection of
isomorphisms
\[
\hocolim_{N\in \sI^\op}\Sigma^{-|N|}_Ts_{q+|N|}^\sR\sM\cong s_q(\sM[Z^{-1}]).
\]
Theorem~\ref{thm:slices of effective spectra} gives us the natural isomorphisms
\[
\Sigma^{-|N|}_Ts_{q+|N|}^\sR\sM\cong  \Sigma^q_T s^\sR_0\sM
\otimes \Z[X]_{q+|N|};
\]
via this isomorphism, the map $\times z_j$ goes over to  $\id_{\Sigma^q_T s^\sR_0\sM}\otimes \times z_j$, 
which yields the result.
\end{proof}

\begin{corollary}\label{cor:slices of quotient spectra}
Let $\sR$, $X$ and $\sM$ be as in theorem~\ref{thm:slices of effective spectra}.
Let $Z=\{z_j\in \Z[X]_{e_j}\}$ be a collection of homogeneous
 elements of $\Z[X]$,   and let $\sM[Z^{-1}]\in \sC$ be the localization
of $\sM$ with respect to the collection of maps $\times z_j:\sM\to
\Sigma^{-e_j}_T\sM$. Let $m\ge2$ be an integer. We let
$\sM[Z^{-1}]/m:=\hocofib\times m:\sM[Z^{-1}]\to \sM[Z^{-1}]$. Then there are
natural isomorphisms
\[
s^\sR_n\sM[Z^{-1}]/m\cong \Sigma^{n}_T s^\sR_0\sM/m
\otimes \Z[X][Z^{-1}]_n,
\]
\[
s_nU\sM[Z^{-1}]/m\cong \Sigma^{n}_T s_0U\sM/m
\otimes \Z[X][Z^{-1}]_n.
\]
\end{corollary}
This follows directly from corollary~\ref{cor:slices of localized spectra},
noting that $s_n^\sR$ and $s_n$ are exact functors.

\begin{remark}\label{rem:Loc} Let $P$ be a multiplicatively closed subset of
$\Z$. We may replace $\Mot$ with its localization $\Mot[P^{-1}]$ with respect to
$P$ in theorem~\ref{thm:slices of effective spectra},  corollary~\ref{cor:slices of localized spectra} and corollary~\ref{cor:slices of quotient spectra}, and obtain a corresponding
description of $s_n^\sR\sM$ and $s_nU\sM$ for a commutative monoid $\sR$ in
$\Mot[P^{-1}]$ and an effective $\sR$-module $\sM$. 

For $P=\Z\setminus\{p^n, n=1, 2,\ldots\}$, we write $\Mot\otimes\Z_{(p)}$ for
$\Mot[P^{-1}]$ and $\SH(S)\otimes\Z_{(p)}$ for $\Ho\Mot\otimes\Z_{(p)}$. 
\end{remark}

\section{The slice spectral sequence}\label{sec:SliceSS}
The slice tower in $\SH(S)$ gives us the slice spectral sequence, for $\sE\in
\SH(S)$, $X\in \Sm/S$, $n\in \Z$,
\begin{equation}\label{eqn:SliceTower0} 
E_2^{p,q}(n):=(s_{-q}(\sE))^{p+q, n}(X) \Longrightarrow \sE^{p+q, n}(X).
\end{equation}
This spectral sequence is not always convergent, however, we do have a
convergence criterion:

\begin{lemma}[\hbox{\cite[lemma 2.1]{LevineExact}}] \label{lem:SliceConv}
Suppose that $S=\Spec k$, $k$ a perfect field. Take $\sE\in \SH(S)$. Suppose
that there is a non-decreasing function $f:\Z\to \Z$ with $\lim_{n\to
\infty}f(n)=\infty$, such that $\pi_{a+b,b}\sE=0$ for $a\le f(b)$. Then the for
all $Y$, and all $n\in\Z$, the spectral sequence \eqref{eqn:SliceTower0} is strongly
convergent.\footnote{As spectral sequence $\{E_r^{pq}\}\Rightarrow G^{p+q}$ {\em
converges strongly} to $G^*$ if for each $n$, the spectral sequence filtration
$F^*G^n$ on $G^n$ is finite and exhaustive, there is an $r(n)$ such that  for
all $p$ and all $r\ge r(n)$, all differentials entering and leaving $E_r^{p,
n-p}$  are zero and the resulting maps $E_r^{p, n-p}\to E_\infty^{p,
n-p}=\gr^p_FG^n$ are all isomorphisms.}  
\end{lemma}

This yields our first convergence result.  For $\sE\in \SH(S)$, $Y\in \Sm/S$,
$p,q,n\in \Z$, define 
\[
H^{p-q}(Y,\pi_{-q}(\sE)(n-q)):=\Hom_{\SH(S)}(\Sigma_T^\infty Y_+, \Sigma^{p+q,
n}s_{-q}(\sE)).
\]
Here $\Sigma^{a,b}$ is suspension with respect to the sphere $S^{a,b}\cong
S^{a-b}\wedge\G_m^{\wedge b}$.
This notation is justified by the case $S=\Spec k$, $k$ a field of
characteristic zero. In this case, there is for each $q$ a canonically defined
object $\pi^\mu_q(\sE)$ of Voevodsky's ``big'' triangulated category of motives
$\DM(k)$, and a canonical isomorphism
\[
\EM(\pi^\mu_q(\sE))\cong \Sigma_T^qs_q(\sE),
\]
where $\EM:\DM(k)\to \SH(k)$ is the motivic Eilenberg-MacLane functor. The
adjoint property of $\EM$ yields the isomorphism
\begin{multline*}
H^{p-q}(Y,\pi^\mu_{-q}(\sE)(n-q)):=\Hom_{\DM(k)}(M(Y),
\pi^\mu_{-q}(\sE)(n-q)[p-q])\\\cong
\Hom_{\SH(S)}(\Sigma_T^\infty Y_+, \Sigma^{p+q, n}s_{-q}(\sE)).
\end{multline*}
We refer the reader to \cite{Pelaez, RO, VoevS0} for details.

\begin{proposition}\label{prop:Conv1} Let $\sR$ be a commutative monoid in $\Mot(S)$, cofibrant as an object in
$\Mot(S)$, with $\sR$ in $\SH^\eff(S)$.  Let $X:=\{\bar{x_i}\in
\sR^{-2d_i,-d_i}(S)\}$ be a countable set of elements of $\sR$-cohomology,  with
$d_i>0$. Let $P$ be a multiplicatively closed subset of $\Z$ and let  $\sM$ be
an $\sR[P^{-1}]$-module, with $U\sM\in \SH(S)^\eff[P^{-1}]$. Suppose that the
canonical map 
\[
U(\sM/(\{x_i\}))\to s_0U\sM
\]
is an isomorphism in $\SH(S)[P^{-1}]$. Then\\
1. The slice spectral sequence for $\sM^{**}(Y)$  has the following form:
\[
E^{p,q}_2(n):=H^{p-q}(Y,\pi^\mu_0(\sM)(n-q))\otimes_\Z\Z[X]_{-q}\Longrightarrow
\sM^{p+q,n}(Y).
\]
2.  Suppose that $S=\Spec k$, $k$ a perfect
field. Suppose further that there is an integer
$a$ such that $\sM^{2r+s, r}(Y)=0$ for all $Y\in \Sm/S$, all $r\in \Z$ and all
$s\ge a$. Then the slice spectral sequence converges strongly
for all $Y\in \Sm/S$, $n\in\Z$.
\end{proposition}

\begin{proof} The form of the slice spectral sequence follows directly from
theorem~\ref{thm:slices of effective spectra}, extended via remark~\ref{rem:Loc}
to the $P$-localized situation. The convergence statement follows directly from
lemma~\ref{lem:SliceConv}, where one uses the function $f(r)=r-a$. 
\end{proof}

We may extend the slice spectral sequence to the localizations $\sM[Z^{-1}]$ as
in corollary~\ref{cor:slices of localized spectra}

\begin{proposition}\label{prop:Conv2} Let $\sR$, $X$, $P$ and $\sM$ be as in proposition~\ref{prop:Conv2} and
that assume  all the hypotheses for (1) in that proposition hold. Let $Z=\{z_j\in \Z[X]_{e_j}\}$ be a collection of homogeneous
 elements of $\Z[X]$,   and let $\sM[Z^{-1}]\in \sC$ be the localization of $\sM$ with
respect to the collection of maps $\times z_j:\sM\to \Sigma^{-e_j}_T\sM$. Then
the slice spectral sequence for $\sM[Z^{-1}]^{**}(Y)$  has the following form:
\[
E^{p,q}_2(n):=H^{p-q}(Y,\pi^\mu_0(\sM)(n-q))\otimes_\Z\Z[X][Z^{-1}]_{-q}
\Longrightarrow \sM[Z^{-1}]^{p+q,n}(Y).
\]
Suppose further that  $S=\Spec k$, $k$ a perfect field, and  there is an integer
$a$ such that $\sM^{2r+s, r}(Y)=0$ for all $Y\in \Sm/S$ all $r\in \Z$ and all
$s\ge a$. Then the slice spectral sequence converges strongly for all
$Y\in \Sm/S$, $n\in\Z$.
\end{proposition}

The proof is same as for proposition~\ref{prop:Conv1}, using
corollary~\ref{cor:slices of localized spectra} to compute the slices of
$\sM[Z^{-1}]$.

\begin{remark}\label{rem:Converge} Let  $\sR$ be a commutative monoid in $\Mot$,
with $\sR\in \SH^\eff(S)$. Suppose that there are elements $a_i\in \sR^{2f_i,
f_i}(S)$, $i=1, 2, \ldots$, $f_i\le0$, so that $\sM$ is the quotient module
$\sR/(\{a_i\})$. Suppose in addition that there is a constant $c$ such that
$\sR^{2r+s, r}(Y)=0$ for all $Y\in \Sm/S$, $r\in \Z$, $s\ge c$. Then 
$\sM^{2r+s, r}(Y)=0$ for all $Y\in \Sm/S$, $r\in \Z$, $s\ge c$. Indeed
\[
\sM:=\hocolim_n\sR/(a_1, a_2,\ldots, a_n)
\]
so it suffices to handle the case $\sM=\sR/(a_1, a_2,\ldots, a_n)$, for which we
may use induction in $n$. Assuming the result for $\sN:=\sR/(a_1, a_2,\ldots,
a_{n-1})$, we have the long exact sequence ($f=f_n$)
\[
\ldots\to \sN^{p+2f,q+f}(Y)\xrightarrow{\times a_n}\sN^{p,q}(Y)\to
\sM^{p,q}(Y)\to \sN^{p+2f+1,q+f}(Y)\to\ldots
\]
Thus the assumption for $\sN$ implies the result for $\sM$ and the induction
goes through. 
\end{remark}

\section{Slices of  quotients of $\MGL$}\label{sec:SlicesTruncatedBPSpectra}
The  slices of a Landweber exact spectrum have
been described by Spitzweck in \cite{Spitzweck10}, but a quotient of $\MGL$ or a localization of such is often not Landweber exact.  We will apply the results of the previous section to describe the slices of the 
motivic truncated Brown-Peterson spectra   $\BP\langle n\rangle$, effective motivic Morava $K$-theory $k(n)$ and 
motivic Morava $K$-theory $K(n)$, as well as recovering the known computations for the Landweber examples \cite{Spitzweck12}, such as the  Brown-Peterson spectra $\BP$ and the Johnson-Wilson spectra $E(n)$. 

Let $\MGL_p$ be the commutative monoid in
$\Mot\otimes\Z_{(p)}$ representing $p$-local algebraic cobordism, as constructed
in \cite[\S 2.1]{PPR}\footnote{This gives $\MGL$ as a symmetric spectrum, we
take the image in the $p$-localized model structure to define $\MGL_p$}.  As
noted in {\it loc.\,cit.}, $\MGL_p$ is a cofibrant object of
$\Mot\otimes\Z_{(p)}$. The motivic $\BP$ was first constructed 
by Vezzosi in \cite{Vezzosi} as a direct summand of $\MGL_p$
by using Quillen's idempotent theorem. Here we construct $\BP$
and $\BP\langle n\rangle$ as quotients of 
$\MGL_p$; the effective Morava $K$-theory is similarly a quotient of $\MGL_p/p$. We consider as well the . Our explicit description of the slices  
allows us to describe the $E_2$-terms of 
slice spectral sequences for $\BP$ and $\BP\langle n\rangle$.

The bigraded coefficient ring $\pi_{*,*}\MGL_p(S)$  
contains $\pi_{2*}MU\simeq \L_*$, localized at $p$, as a graded subring of the bi-degree  $(2*,*)$ part, via the classifying map for the formal group law of $\MGL$, split by the appropriate Betti or \'etale realization map. The
ring $\mathbb{L}_{*p}:=\L_*\otimes_\Z\Z_{(p)}$ is isomorphic to polynomial ring 
$\Z_{(p)}[x_1, x_2, \cdots]$ \cite[Part II, theorem 7.1]{Adams}, where the element $x_i$ 
has degree $2i$ in $\pi_{*}MU$,  degree $(2i, i)$ in $\pi_{*,*}MGL_p$ and degree $i$ in $\L_*$. 

The following result of Hopkins-Morel-Hoyois \cite{Hoyois} is crucial for the application of the general results of the previous sections to quotients of $\MGL$ and $\MGL_p$.

\begin{theorem}[\hbox{\cite[theorem 7.12]{Hoyois}}]\label{thm:HMH} Let $p$ be a prime integer, $S$ an essentially smooth scheme over a field of characteristic prime to $p$. Then the canonical maps $\MGL_p/(\{x_i:i=1, 2,\ldots\})\to s_0\MGL_p\to H\Z_{(p)}$ are isomorphisms in $\SH(S)$.  In case $S=\Spec k$, $k$ a perfect field of characteristic prime to $p$, the inclusion $\L_{*p}\subset\pi_{2*,*}\MGL_p(S)$  is an equality.
\end{theorem}

We define a series of subsets of the set of generators $\{x_i\ |\ i=1,
2\ldots\}$, 
\begin{align*}
&B^c_p=\{x_i:i \neq p^k-1, k\ge1 \},\\
&B_p=\{x_i:i = p^k-1, k\ge1 \},\\
&B\langle n \rangle^c_p= \{x_i:i \neq p^k-1, 1\le k\leq n \}, \\
&B\langle n \rangle_p= \{x_i:i = p^k-1, 1\le k\leq n \},\\
&k\langle n \rangle_p= \{x_{p^n-1}\}.
\end{align*}
We also define 
\[
k\langle n \rangle^c_p= \{x_i : i \neq p^n-1, \mathrm{and}\ x_0=p\}\subset \{p, x_i\ |\ i=1,
2\ldots\}.
\]

\begin{definition}[$BP,\ BP\langle n\rangle$ and $E(n)$]
\label{BP, BP<n> and E(n)}
The  Brown-Peterson spectrum $\BP$ is defined as
\[
\BP:=\MGL_p/(\{x_i\ |\ i\in B^c_p\}),
\]
the truncated Brown-Peterson spectrum $\BP\langle n\rangle$
is  defined as 
\[
\BP\langle n\rangle:=\MGL_p/(\{x_i\ |\ i\in B\langle n \rangle^c_p\})
\]
and the Johnson-Wilson spectrum $E(n)$ is the localization
\[
E(n):=\BP\langle n\rangle[x_{p^n-1}^{-1}].
\]
\end{definition}

\begin{definition}[Morava $K$-theories $k(n)$ and $K(n)$]
\label{k(n) and K(n)}
{\em Effective Morava $K$-theory} $k(n)$ defined as
\[
k(n):=\MGL_p/(\{x_i\ |\ i\in k\langle n \rangle^c_p\})\cong \BP\langle n\rangle/(x_{p-1}, \ldots, x_{p^{n-1}-1},p).
\]
Define Morava $K$-theory $K(n)$ as the localization
\[
K(n):= k(n)[x_{p^n-1}^{-1}]
\]
\end{definition}

The spectra $\BP, \BP\langle n\rangle,\mm E(n), k(n)$ 
and $K(n)$ are $\MGL_p$-modules.  $\BP$ and $E(n)$ are Landweber exact. We let
$\sC$ denote the category of $\MGL_p$-modules.

\begin{lemma}
\label{lem:effectivity of various BP-related quotients}
The $MGL_p$-module spectra $\BP, \BP\langle n\rangle$ and  $k(n)$ 
 are effective.  $\BP$ and $E(n)$ have the structure of  oriented weak
commutative ring $T$-spectra in $\SH(S)$.
\end{lemma}

\begin{proof}
The effectivity of these theories follows from lemma
\ref{effectivity of quotients of effective spectra} and the fact that homotopy
colimits of effective spectra are effective. The ring structure for $\BP$ and
$E(n)$ follows from the Landweber exactness (see \cite{NSO}).
\end{proof}

We first discuss the effective theories $\BP, \BP\langle n\rangle$ and  $k(n)$.

\begin{proposition}
\label{prop:zeroth slices of BP and BP<n>} Let $p$ be a prime, $k$  a  field with
exponential characteristic prime to $p$ and $S$ an essentially smooth $k$-scheme.  Then in $\SH(S)$:\\
1. The zeroth slices of both
$\BP$ and $\BP\langle n\rangle$ are isomorphic to $p$-local  motivic
Eilenberg-MacLane spectrum $H\Z_{(p)}$, and the zeroth slice 
of $k(n)$ is isomorphic to $H\Z/p$.\\
2. The quotient maps from $\MGL_p$ induce isomorphisms
\[
s_0\BP\simeq  (s_0\MGL)_p\simeq s_0\BP\langle n\rangle
\]
and 
\[
s_0k(n)\simeq   (s_0\MGL)_p/p.
\]
3. The respective quotient maps from $\BP$, $\BP\langle n\rangle$ and $k(n)$
induce isomorphisms
\begin{align*}
&\BP/(\{x_i : x_i\in B_p\})  \simeq s_0\BP\\
&\BP\langle n\rangle/(\{x_i: x_i\in B\langle n\rangle_p\}) \simeq s_0\BP\langle
n\rangle\\
&k(n)/(x_{p^n-1})\simeq s_0k(n)
\end{align*}
\end{proposition}

\begin{proof}
By theorem~\ref{thm:HMH},  the classifying map $\MGL\to H\Z$ for
motivic cohomology induces isomorphisms
\[
\MGL_p/(\{x_i:i=1, 2,\ldots\})\cong s_0\MGL_p\cong H\Z_{(p)}
\]
in $\SH(S)\otimes\Z_{(p)}$.

Now let $\sS\subset \N$ be a subset and $\sS^c$ its complement. By
remark~\ref{rem:DoubleQuot}, we have an isomorphism
\[
(\MGL_p/(\{x_i:i\in \sS^c\}))/(\{x_i:i\in \sS\})\cong \MGL_p/(\{x_i:i\in\N\}).
\]
Also, as $x_i$ is a map $\Sigma^{2i,i}\MGL_p\to \MGL_p$,  $i>0$, the quotient
map $\MGL_p\to \MGL_p/(\{x_i:i\in \sS^c\})$ induces an isomorphism
\[
s_0\MGL_p\to s_0[\MGL_p/(\{x_i:i\in \sS^c\})].
\]
This gives us isomorphisms
\[
(\MGL_p/(\{x_i:i\in \sS^c\}))/(\{x_i:i\in \sS\})\cong s_0[\MGL_p/(\{x_i:i\in
\sS^c\})]\cong s_0\MGL_p 
\]
with the first isomorphism induced by the quotient map
\[
\MGL_p/(\{x_i:i\in \sS^c\}\to (\MGL_p/(\{x_i:i\in \sS^c\}))/(\{x_i:i\in \sS\}).
\]
Taking $\sS=B_p, B\langle n\rangle_p, \{x_{p^n-1}\}$ proves the result for
$\BP$, $\BP\langle n\rangle$ and $k(n)$, respectively. 
\end{proof}

For motivic spectra $\sE= \BP,\ \BP\langle n\rangle,\ k(n),\ E(n)
\ \mathrm{and}\ K(n)$ defined in \ref{BP, BP<n> and E(n)} and
\ref{k(n) and K(n)} let us denote the corresponding
topological spectra by $\sE^{top}$. 
The graded coefficient rings $\sE^{top}_*$
of these topological spectra are
\[
\sE^{top}_* \simeq  \left\{\begin{array}{ll}  \Z_p[v_1,\mm 
                  v_2,\cdots]&\ \ \ \sE= \BP\\
                  \Z_p[v_1,\mm v_2, \cdots,v_n]
                   &\ \ \ \sE=\BP\langle n\rangle \\
                  \Z_p[v_1,\mm v_2, \cdots,v_n, \mm v_n^{-1}]
                   &\ \ \ \sE=E(n)\\
                   \Z/p[v_n]
                   &\ \ \ \sE=k(n)\\
                   \Z/p[v_n, \mm v_n^{-1}]
                   &\ \ \ \sE=K(n)                                        
\end{array} \right\}
\]
where deg\mm$v_n=2(p^n-1)$. The element $v_n$ corresponds to the element
$\bar{x}_n\in \MGL^{2n,n}(k)$.

\begin{corollary}
\label{cor:slices of BP and BP<n>}   Let $p$ be a prime, $k$  a  field with
exponential characteristic prime to $p$ and $S$ an essentially smooth $k$-scheme. 
Then in $\SH(S)$, the slices of Brown-Peterson, Johnson-Wilson and Morava theories are given by
\[
s_i\sE \simeq  \left\{\begin{array}{ll} \Sigma^i_T\mm 
                                H_{\Z_p}\otimes \sE^{top}_{2i}   
                                &\ \ \ \sE= \BP,\ \BP\langle n\rangle 
                                \ \mathrm{and}\ E(n)\\
                                 \Sigma^i_T\mm 
                                H_{\Z/p}\otimes \sE^{top}_{2i}   
                                & \ \ \ \sE= 
                                k(n)\ \mathrm{and}\ K(n)                        
               
\end{array} \right\}
\]
where $\sE^{top}_{2i}$ is degree $2i$ homogeneous component
of coefficient ring of the corresponding topological theory.
\end{corollary}

\begin{proof}
The statement for $\BP$ and  $\BP\langle n\rangle$  follows 
from theorem~\ref{thm:slices of effective spectra}, and remark~\ref{rem:Loc}.
The case of   $E(n)$ follows from corollary~\ref{cor:slices of localized
spectra} and the cases of $k(n)$ and 
$K(n)$  follow from corollary~\ref{cor:slices of quotient spectra}.
\end{proof}

\begin{theorem}\label{slice spectral sequence for BP and BP<n>}  Let $p$ be a prime, $k$  a  field with
exponential characteristic prime to $p$ and $S$ an essentially smooth $k$-scheme. 
The slice spectral sequence for any of the spectra 
$\sE= \BP,\ \BP\langle n\rangle,\ k(n),\ E(n)
\ \mathrm{and}\ K(n)$ in $\SH(S)$ has the form 
\[
\sE^{p,q}_2(X,m)= H^{p-q}(X, 
\mathcal{Z}(m-q))\otimes_\Z \sE^{top}_{-2q}
\Rightarrow \sE^{p+q, m}(X)
\]
where $\mathcal{Z}= \Z_p$ for $\sE=\BP,\ \BP\langle n\rangle$
and $E(n)$, and $\mathcal{Z}= \Z/p$ for $\sE=k(n)$
and $K(n)$. In case $S=\Spec k$ and $k$ is perfect, these spectral sequences are all strongly convergent.
\end{theorem}

\begin{proof} The form of the
slice spectral sequence for $\sE$ follows from corollary~\ref{cor:slices of BP and BP<n>}.
The fact that the slice spectral sequences strongly converge for $S=\Spec k$, $k$ perfect, follows from
remark~\ref{rem:Converge} and the fact that $\MGL^{2r+s,r}(Y)=0$ for all
$Y\in\Sm/S$, $r\in\Z$ and $s\ge1$. This in turn follows from the
Hopkins-Morel-Hoyois spectral sequence
\[
E_2^{p,q}(n):=H^{p-q}(Y,\Z(n-q))\otimes\L_{-q}\Longrightarrow \MGL^{p+q,n}(Y)
\]
which is strongly convergent by \cite[theorem 8.12]{Hoyois}. 
\end{proof}

\section{Modules for oriented theories}\label{sec:modules}
We will use the slice spectral sequence to compute the ``geometric
part'' $\sE^{2*,*}$ of a quotient spectrum $\sE=\MGL_p/(\{x_{i_j}\})$ in terms
of algebraic cobordism, when working over a base field $k$ of characteristic
zero. As the quotient spectra are naturally $\MGL_p$-modules but may not have a
ring structure, we will need to extend the existing theory of oriented
Borel-Moore homology and related structures to allow for modules over ring-based
theories.

\subsection{Oriented Borel-Moore homology}
We first discuss the extension of oriented Borel-Moore homology. We use the
notation of \cite{LM}.  Let $\Sch/k$ be the category of quasi-projective schemes
over a field $k$ and let $\Sch/k'$ denote the subcategory of projective morphisms in 
$\Sch/k$. Let $\Ab_*$ denote the category of graded abelian groups, $\Ab_{**}$ the category of bi-graded abelian groups. 

\begin{definition} Let $A$ be an oriented Borel-Moore homology theory on
$\Sch/k$. An oriented $A$-module $B$ is given by
\\\\
(MD1) An additive functor $B_*:\Sch/k'\to {\bf Ab}_*$, $X\mapsto B_*(X)$.\\
(MD2) For each \lci morphism $f:Y\to X$ in $\Sch/k$ of relative dimension $d$, a
homomorphism of graded groups $f^*:B_*(X)\to B_{*+d}(Y)$.\\
(MD3) For each pair $(X,Y)$ of objects in $\Sch/k$ a bilinear graded pairing
\begin{align*}
A_*(X)\otimes B_*(Y)&\to B_*(X\times_kY)\\
u\otimes v&\mapsto u\times v
\end{align*}
which is associative and unital with respect to the external products in the
theory $A$. 

These satisfy the conditions (BM1), (BM2), (PB) and (EH) of \cite[definition
5.1.3]{LM}. In addition, these satisfy the following modification of (BM3)
\\\\
(MBM3) Let $f:X'\to X$ and $g:Y'\to Y$ be morphisms in $\Sch/k$. If $f$ and $g$
are projective, then for $u'\in A_*(X')$, $v'\in B_*(Y')$, one has
\[
(f\times g)_*(u'\times v')=f_*(u')\times g_*(v').
\]
If $f$ and $g$ are \lci morphisms, then for $u\in A_*(X)$, $v\in B_*(Y)$, one
has
\[
(f\times g)^*(u\times v)=f_*(u)\times g_*(v).
\]
\ \\
Let $f:A\to A'$ be a morphism of Borel-Moore homology theories, let $B$ be an
oriented $A$-module, $B'$ an oriented $A'$-module. A morphism $g:B\to B'$ over
$f$ is a collection of homomorphisms of graded abelian groups $g_X:B_*(X)\to
B'_*(X)$, $X\in \Sch/k$ such that the $g_X$ are compatible with projective
push-forward, \lci pull-back and external products.
\end{definition}
We do not require the analog of the axiom (CD) of \cite[definition
5.1.3]{LM}; this axiom plays a role only in the
proof of universality of $\Omega_*$, whereas the universality of $\Omega$ for
$A$-modules follows formally from the universality for $\Omega$ among oriented
Borel-Moore homology theories (see proposition~\ref{prop:universality} below).

\begin{example} Let $N_*$ be a graded module for the Lazard ring $\L_*$ and let
$A_*$ be an oriented Borel-Moore homology theory. Define
$A_*^N(X):=A_*(X)\otimes_{\L_*}N_*$. Then with push-forward
$f^N_*:=f^A_*\otimes\id_{N_*}$, pull-back   $f_N^*:=f_A^*\otimes\id_{N_*}$, and
product $u\times(v\otimes n):=(u\times v)\otimes n$, for $u\in A_*(X)$, $v\in
A_*(Y)$, $n\in N_*$, $A_*^N$ becomes an oriented $A$ module. Sending $N_*$ to
$A_*^N$ gives a functor from graded $\L_*$-modules to oriented $A$-modules.

In case $k$ has characteristic zero, we note that, for $A_*=\Omega_*$, we have a canonical isomorphism
$\theta_{N_*}:\Omega_*^{N_*}(k)\cong N_*$, as the classifying map $\L_*\to \Omega_*(k)$ is an
isomorphism \cite[theorem 1.2.7]{LM}.
\end{example} 

Just as for a Borel-Moore homology theory, one can define operations of $A_*(Y)$
on $B_*(Z)$ via a  morphism $f:Z\to Y$, assuming that $Y$ is in $\Sm/k$: for
$a\in A_*(Y)$, $b\in B_*(Z)$, define $a\cap_fb\in B_*(Z)$ by 
\[
a\cap_fb:=(f,\id_Z)^*(a\times b)
\]
where $(f,\id_Z):Z\to Y\times_kZ$ is the (transpose of) the graph embedding. As
$Y$ is smooth over $k$, $(f,\id_Z)$ is an \lci morphism, so the pullback
$(f,\id_Z)^*$ is defined. Similarly, $B_*(Y)$ is an $A_*(Y)$-module via
\[
a\cup_Yb:=\delta_Y^*(a\times b).
\]
These products satisfy the analog of the properties listed in \cite[\S
5.1.4, proposition 5.2.1]{LM}.

\begin{proposition}\label{prop:universality} Let $A$ be an oriented Borel-Moore
homology theory on $\Sch/k$ and let $B$ be an oriented $A$-module. Let 
$\vartheta_A :\Omega_*\to A_*$ be the classifying map. There is a unique
morphism $\theta_{A/B}:\Omega_*^{B_*(k)}\to B_*$ over $\vartheta_A$ such that
$\theta_{A/B}(k):\Omega_*^{B_*(k)}(k)\to B_*(k)$ is the canonical isomorphism
$\theta_{B_*(k)}$.
\end{proposition}

\begin{proof}
For $X\in \Sch/k$, $b\in B_*(k)$  and  $u\in \Omega_*(X)$, we define
$\theta_{A/B}(u\otimes b):=\vartheta_A(u)\times b\in B_*(X\times_kk)=B_*(X)$. It
is easy to check that this defines a morphism over $\vartheta_A$. Uniqueness
follows easily from the fact that the product structure in $A$ and $\Omega$ is
unital.
\end{proof}

\subsection{Oriented duality theories}
Next, we discuss a theory of modules for an oriented duality theory $(H, A)$. We
use the notation and definitions from \cite{LevineOrient}. In particular, we have the category $\SP$ of smooth pairs over $k$, with objects $(M, X)$, $M\in \Sm/k$, $X\subset M$ a closed subset, and where a morphism $f:(M,X)\to (N,Y)$ is a morphism $f:M\to N$ in $\Sm/k$ such that $f^{-1}(Y)\subset X$.

\begin{definition} Let $A$ be a bi-graded oriented ring cohomology theory, in
the sense of \cite[definition 1.5, remark 1.6]{LevineOrient}. An oriented
$A$-module $B$ is a bi-graded cohomology theory on $\SP$, satisfying the analog
of \cite[definition 1.5]{LevineOrient}, that is:
for each pair of smooth pairs $(M, X)$, $(N,Y)$ there is a bi-graded
homomorphism
\[
\times: A^{**}_X(M)\otimes B^{**}_Y(N)\to B^{**}_{X\times Y}(M\times_kN)
\]
satisfying
\begin{enumerate}
\item {\em associativity}: $(a\times b)\times c=a\times(b\times c)$ for $a\in
A^{**}_X(M)$, $b\in A^{**}_Y(N)$, $c\in B^{**}_Z(P)$
\item {\em unit}: $1\times a=a$.
\item {\em Leibniz rule}: Given smooth pairs $(M,X)$, $(M, X')$, $(N,Y)$ with
$X\subset X'$ we have
\[
\partial_{M\times N, X'\times N, X\times N}(a\times b)=\partial_{M, X',
X}(a)\times b
\]
for $a\in A^{**}_{X'\setminus X}(M\setminus X)$, $b\in B^{**}_Y(N)$. For  a
triple $(N, Y', Y)$ with $Y\subset Y'\subset N$,  $a\in A^{m,*}_{X'}(M)$, $b\in
B_{Y'\setminus Y}(N\setminus Y)$ we have
\[
\partial_{M\times N, M\times Y', M\times Y}(a\times b)=(-1)^m
a\times\partial_{N, Y', Y}(b).
\]
\end{enumerate}
We write $a\cup b\in B_{X\cap Y}(M)$ for $\delta^*_M(a\times b)$, $a\in
A^{**}_X(M)$, $b\in B^{**}_Y(M)$.

In addition, we assume that the ``Thom classes theory'' \cite[lemma
3.7.2]{Panin} arising from the orientation on $A$ induces an orientation on $B$
in the following sense: Let $(M, X)$ be a smooth pair and let $p:E\to M$ be a
rank $r$ vector bundle on $M$. Then the cup product with the Thom class
$th(E)\in A^{2r,r}_M(E)$
\[
B^{**}_X(M)\xrightarrow{p^*}B^{**}_{p^{-1}(X)}(E)\xrightarrow{th(E)\cup(-)}B^{
2r+*, r+*}_X(E)
\]
is an isomorphism.
\end{definition}

Let $\SP'$ be the category with the same objects $(M,X)$ as in $\SP$, and where a morphism $f:(M,X)\to (N,Y)$ is a projective morphism $f:M\to N$ such that $f(X)\subset Y$. One proceeds just as in \cite{LevineOrient} to show that the orientation on $B$
gives rise to an integration on $B$,  that is,  one has for each morphism
$F:(M,X)\to (N,Y)$ in $\SP'$ a pushforward map $F_*:B^{**}_X(M)\to B^{*-2d,
*-d}_Y(N)$, $d=\dim_kM-\dim_kN$, defining an {\em integration with supports} for
$B$, in the sense of \cite[definition 1.8]{LevineOrient}, with the introduction
of the bi-grading and the evident change to definition 1.8(2), in that the product
$f^*(-)\cup$ is a map from $A^{**}_Z(M)\otimes B^{**}_Y(N)$ to $B_{Y\cap
f^{-1}(Z)}(N)$, and $\cup$ is  similarly a map from $A^{**}_Z(M)\otimes
B^{**}_X(M)$ to $B_{X\cap Z}(M)$.  One similarly proves the analog of
\cite[theorem 1.12]{LevineOrient}, that the integration so constructed is the
unique integration on $B$ subjected to the orientation induced by the
orientation on $A$.

\begin{definition} Let $(H, A)$ be an oriented duality theory, in the sense of
\cite[definition 3.1]{LevineOrient}. An oriented $(H,A)$-module is a pair $(J, B)$,
where\\
(D1)  $J:\Sch/k'\to \Ab_{**}$ is a functor \\
(D2)  $B$ is an oriented $A$-module, \\
(D3) For each open immersion $j:U\to X$  there is a pullback map
$j^*:J_{**}(X)\to J_{**}(U)$ \\
(D4) i.   for each smooth pair $(M,X)$ and each morphism $f:Y\to M$ in $\Sch/k$
and bi-graded cap product map
\[
f^*(-)\cap:A_X(M)\otimes H(Y)\to H(f^{-1}(X))
\]
ii. For $X, Y\in \Sch/k$ a bi-graded external product
\[
\times: H_{**}(X)\otimes J_{**}(Y)\to J_{**}(X\times Y).
\]
(D5) For each smooth pair $(M, X)$, a graded isomorphism
\[
\beta_{M,X}:J_{**}(X)\to B_X^{2d-*, d-*}(M);\quad d=\dim_kM.
\]
(D6) For each $X\in \Sch/k$ and each closed subset $Y\subset X$, a   map
\[
\partial_{X,Y}:J_{*+1,*}(X\setminus Y)\to J_{**}(Y).
\]
These satisfy the evident analogs of properties (A1)-(A4) of \cite[definition
3.1]{LevineOrient}, where we make the following changes: Let  $d=\dim_kM$, $e=\dim_kN$. One  replaces $H$ with
$J_{**}$   throughout (except in (A3)(ii)), and
\begin{itemize}
\item in (A1) one replaces   $A_Y(N)$, $A_X(M)$ with $B_Y^{2d-*,d-*}(N)$,
$B_X^{2d-*,d-*}(M)$,
\item  in (A2) on replaces $A_Y(N)$, $A_X(M)$ with $B_Y^{2e-*,e-*}(N)$,
$B_X^{2d-*,d-*}(M)$, 
\item  in (A3)(i) one replaces $A_Y(M)$ with $B_Y^{2d-*, d-*}(M)$ and $A_{Y\cap
f^{-1}(X)}(N)$\\
with $B^{2e-*, e-*}_{Y\cap f^{-1}(X)}(N)$, 
\item in (A3)(ii) one replaces  $A_Y(M)$ with $B_Y^{2e-*, e-*}(N)$ and 
$A_{X\times Y}(M\times N)$ \\
 with $B^{2(d+e)-*, d+e-*}_{X\times Y}(M\times N)$, $H(X)$ with $H_{**}(X)$,
$H(Y)$ with $J_{**}(Y)$ \\
 and $H(X\times Y)$ with $J_{**}(X\times Y)$.
\item in (A4) one replaces $A_{X\setminus Y}(M\setminus Y)$ with $B^{2d-*,
d-*}_{X\setminus Y}(M\setminus Y)$.
\end{itemize}
\end{definition}

\begin{remark}  \label{rem:Univ} Let  $(H, A)$ be an oriented duality theory on
$\Sch/k$, for $k$ a field admitting resolution of singularities.  By
\cite[proposition 4.2]{LevineOrient} there is a unique natural transformation 
\[
\vartheta_H:\Omega_*\to H_{2*,*}
\]
of functors $\Sch/k'\to \Ab_*$  compatible with all the structures available
for $H_{2*,*}$ and, after restriction to $\Sm/k$ is just the classifying map
$\Omega^*\to A^{2*,*}$ for the oriented cohomology theory $X\mapsto
A^{2*,*}(X)$. We refer the reader to \cite[\S 4]{LevineOrient} for a complete
description of the properties satisfied by $\vartheta_H$.

Via $\vartheta_H$ and the ring homomorphism $\rho_\Omega:\L_*\to \Omega_*(k)$
classifying the formal group law for $\Omega_*$, we have the ring homomorphism
$\rho_H:\L_*\to H_{2*,*}(k)$. If $(J, B)$ is an  oriented $(H, A)$-module, then
via the $H_{2*,*}(k)$-module structure on $J_{2*,*}(k)$, $\rho_H$ makes 
$J_{2*,*}(k)$ a $\L_*$-module. We write $J_*$ for the $\L_*$-module
$J_{2*,*}(k)$.
\end{remark}

\begin{proposition}\label{prop:Univ2} Let $k$ be a field admitting resolution of
singularities. Let  $(H, A)$ be an oriented duality theory and $(J, B)$ an 
oriented $(H, A)$-module. 
There is a unique natural transformation $\vartheta_{H/J}:\Omega_*^{J_*}\to
J_{2*,*}$ from $\Sch/k'\to \Ab_*$   satisfying
\begin{enumerate}
\item $\vartheta_{H/J}$ is compatible with pullback maps $j^*$ for $j:U\to X$ an
open immersion in $\Sch/k$.
\item $\vartheta_{H/J}$ is compatible with fundamental classes.
\item $\vartheta_{H/J}$ is compatible with external products.
\item $\vartheta_{H/J}$ is compatible with the action of 1st Chern class
operators.
\item Identifying $\Omega_*^{J_*}(k)$ with $J_{2*,*}(k)$ via the product
map $\Omega_*(k)\otimes_{\L_*}J_{2*,*}(k)\to J_{2*,*}(k)$, 
$\vartheta_{H/J}(k):\Omega_*^{J_*}(k)\to J_{2*,*}$ is  the identity map.
\end{enumerate}
\end{proposition}

\begin{proof} For $X\in\Sch/k$, we define $\vartheta_{H/J}(X)$ by
\[
\vartheta_{H/J}(u\otimes j)=\vartheta_H(u)\times j\in J_{2*,*}(X\times_k\Spec
k)=J_{2*,*}(X),
\]
for $u\otimes j\in \Omega_*^{J_*}(X):=\Omega_*(X)\otimes_{\L_*}J_{2*,*}(k)$. The
properties (1)-(5) follow directly from the construction. As $\Omega_*(X)$ is
generated by push-forwards of fundamental classes, the properties (2), (3) and
(5) determine $\vartheta_{H/J}$ uniquely.
\end{proof}

\begin{remark} \label{rem:products} Let $k$, $(H,A)$ and $(J,B)$ be as in
proposition~\ref{prop:Univ2}. Suppose that $J_*:=J_{2*,*}$ has external products
$\times_J$ and there is a unit element $1_J\in J_0(k)$ for these external
products. Suppose further that these are compatible with the external products
$H_*(X)\otimes J_*(Y)\to J_*(X\times_kY)$, in the sense that
\[
(h\times 1_J)\times_J b=h\times b\in J_*(X\times_kY)
\]
for $h\in H_*(X)$, $b\in J_*(Y)$, and that $1_H\times 1_J=1_J$.  Then
$\vartheta_{H/J}$ is compatible with external products and is unital. This
follows directly from our assumptions and the identity
\[
\vartheta_{H/J}((u\otimes h)\times(u'\otimes j'))=\vartheta_H(u)\times
\vartheta_{H/J}(u'\otimes (h\times j)).
\]
\end{remark}

\subsection{Modules for oriented ring spectra}
We now discuss the oriented duality theory and oriented Borel-Moore homology
associated to a module spectrum for an oriented weak commutative ring
$T$-spectrum.

Let $\ph$ be the two-sided ideal of phantom maps in $\SH(S)$, that is a map $f:\sE\to \sF$
that vanishes after pre-composition with a map from a compact object. 
Let $\sE$ be a weak commutative ring $T$-spectrum, that is, there are maps $\mu:\sE\wedge\sE\to \sE$, $\eta:\mS_S\to \sE$ in $\SH(S)$ that satisfy the axioms for a monoid modulo phantom maps.  An $\sE$-module  is similarly an object $\sN\in \SH(S)$ together
with a multiplication map $\rho:\sE\wedge\sN\to \sE$ that makes $\sN$ into a
unital $\sE$-module modulo phantoms.

Suppose that $(\sE, c)$ is an oriented weak commutative ring $T$-spectrum in
$\SH(k)$, $k$ a field admitting resolution of singularities. We have constructed
in \cite[theorem 3.4]{LevineOrient} a bi-graded oriented duality theory
$(\sE'_{**}, \sE^{**})$ by defining $\sE'_{a,b}(X):=\sE_X^{2m-a, m-b}(M)$, where $M\in\Sm/k$ is a chosen smooth quasi-projective scheme
containing $X$ as a closed subscheme and 
$m=\dim_kM$. Let $\sN$ be an $\sE$-module. For $E\to M$
a rank $r$ vector bundle on $M\in \Sm/k$ and $X\subset M$ a closed subscheme,
the Thom classes for $\sE$ give rise to a Thom isomorphism $\sN^{**}_X(M)\to
\sN^{2r+*, r+*}_X(E)$.

Using these Thom isomorphisms, the arguments used to construct the oriented
duality theory $(\sE'_{**}, \sE^{**})$ go through without change to give $\sN^{**}$ the
structure of an oriented $\sE^{**}$-module, and to define an oriented $(\sE'_{**}, \sE^{**})$-module $(\sN'_{**}, \sN^{**})$, with 
canonical isomorphisms $\sN'_{a,b}(X)\cong \sN_X^{2m-a, m-b}(M)$, $m=\dim_kM$, and where the cap products are induced by the $\sE$-modules structure on $\sN$.

\subsection{Geometrically Landweber exact modules}

\begin{definition} Let $(\sE, c)$ be a weak oriented ring $T$-spectrum and let
$\sN$ be an $\sE$-module. The {\em geometric part} of $\sE^{**}$ is the $(2*,*)$-part $\sE^*:=\sE^{2*,*}$ of $\sE^{**}$, the geometric part of $\sN$ is the  
$\sE^*$-module $\sN^{2*,*}$, and the geometric part of
$\sN'$ is similarly given by $X\mapsto \sN'_*(X):=\sN'_{2*,*}(X)$. This gives us the $\Z$-graded oriented duality theory $(\sE'_*, \sE^*)$ and the oriented $(\sE'_*, \sE^*)$-module $(\sN'_*, \sN^*)$.
\end{definition}

Let $(\sE, c)$ be a weak oriented ring $T$-spectrum and let
$\sN$ be an $\sE$-module. By proposition~\ref{prop:Univ2}, we have a canonical natural transformation 
\[
\vartheta_{\sE'/\sN'}:\Omega_*^{\sN'_*(k)}\to \sN'_*
\]
satisfying the compatibilities listed in that proposition. 

We extend the definition of a geometrically Landweber exact weak commutative
ring $T$ -spectrum (see \cite[definition 3.7]{LevineExact}) to the case of an
 $\sE$-module:

\begin{definition}\label{def:GeomLE} Let $(\sE, c)$ be a weak oriented ring
$T$-spectrum and let $\sN$ be an $\sE$-module. We say that $\sN$ is {\em 
geometrically Landweber exact} if  for each point $\eta\in X\in\Sm/k$\\\\
i.  The structure map $p_\eta:\eta\to \Spec k$ induces an isomorphism
$p_\eta^*:\sN^{2*, *}(k)\to \sN^{2*,*}(\eta)$.\\
ii. The product map $\cup_\eta: \sE^{1,1}(\eta)\otimes\sN^{2*,*}(\eta)\to
\sN^{2*+1, *+1}(\eta)$ induces a surjection $k(\eta)^\times\otimes
\sN^{2*,*}(\eta)\to \sN^{2*+1, *+1}(\eta)$
\\\\
Here we use the canonical natural transformation $t_\sE:\G_m\to \sE^{1,1}(-)$
defined in \cite[remark 1.5]{LevineExact} to define the map $k(\eta)^\times\to
\sE^{1,1}(\eta)$ needed in (ii).
\end{definition}

The following result generalizes \cite[theorem 6.2]{LevineExact} from oriented
weak commutative ring $T$-spectra to modules:

\begin{theorem} \label{thm:Rational} Let $k$ be a field of characteristic zero,
$\sN$ an $\MGL$-module in $\SH(k)$, $(\sN'_{**}, \sN^{**})$ the associated oriented $(\MGL'_{**},
\MGL^{**})$-module, and $\sN'_*$ the geometric part of $\sN'$. Suppose that $\sN$ is
geometrically Landweber exact.  Then the classifying map
\[
\vartheta_{\MGL'_*/\sN'_*}:\Omega_*^{\sN'_*(k)}\to \sN'_*
\]
is an isomorphism.
\end{theorem}

\begin{remark}  Let $k$ be a field of characteristic zero, and let $(\sE, c)$ be an oriented weak commutative ring $T$-spectrum
in $\SH(S)$, and let $\sN$ be an $\sE$-module. Via the classifying map
$\phi_{\sE,c}:\MGL\to \sE$, $\sN$ becomes an $\MGL$-module. In addition,  
the classifying map $\vartheta_{\sE'}:
\Omega_*\to \sE'_*$ is induced from $\phi_{\sE,c}$ and the classifying map
$\vartheta_{\MGL'_*/\sN'_*}$ factors through the classifying map
$\vartheta_{\sE'_*/\sN'_*}:\sE_*^{\prime \sN'_*(k)}\to \sN'_*$ as
\[
\vartheta_{\MGL'_*/\sN'_*}=\vartheta_{\sE'_*/\sN'_*}\circ
(\phi_{\sE,c}\otimes\id_{\sN'_*(k)}).
\]
Thus, theorem~\ref{thm:Rational} applies to $\sE$-modules for arbitrary $(\sE,
c)$. Moreover, if $(\sE,c)$ is geometrically Landweber exact in the sense of
\cite[definition 3.7]{LevineExact}, the map
$\bar{\vartheta}_{\sE'_*}:\Omega_*^{\sE'_*(k)}\to \sE'_*$ is an isomorphism
(\cite[theorem 6.2]{LevineExact}) hence the map $\vartheta_{\sE'_*/\sN'_*}$ is
an isomorphism as well.
\end{remark}

\begin{proof}[proof of theorem~\ref{thm:Rational}]
The proof of theorem~\ref{thm:Rational} is essentially the same as the proof of
\cite[theorem 6.2]{LevineExact}. Indeed, just as in {\it loc. cit.}, one
constructs a commutative diagram (see \cite[(6.4)]{LevineComp})
\begin{equation}\label{eqn:Diag3}
\xymatrixcolsep{14pt}
\xymatrix{
\oplus_{\eta\in X_{(d)}}k(\eta)^\times \otimes \sN'_{*-d+1}
\ar[r]^-{\Div_\sN}&\Omega_*^{\sN'_*(1)}(X)\ar[d]_{\bar\vartheta^{(1)}}\ar[r]^-{
i_*}&\Omega^{\sN'_*}_*(X)\ar[d]_{\bar\vartheta(X)}\ar[r]^-{j^*}&\oplus_{\eta\in
X_{(d)}}\Omega^{\sN'_*}_*(\eta)\ar[r]\ar[d]_{\bar\vartheta}&0\\
\oplus_{\eta\in X_{(d)}}k(\eta)^\times\otimes \sN'_{*-d+1} \ar[r]_-{div_\sN}
&\sN_{2*,*}^{\prime
(1)}(X)\ar[r]_-{i_*}&\sN_{2*,*}'(X)\ar[r]_-{j^*}&\oplus_{\eta\in
X_{(d)}}\sN_{2*,*}'(\eta)\ar[r]&0}
\end{equation}
where we write $\sN'_*$ for $\sN'_*(k)$, $d$ is the maximum of $\dim_kX_i$ as
$X_i$ runs over the irreducible components of $X$, and $\sN_{2*,*}^{\prime
(1)}(X)$ is the colimit of $\sN_{2*,*}'(W)$, as $W$ runs over closed subschemes
of $X$ containing no dimension $d$ generic point of $X$. A similarly defined
colimit of the $\Omega_*^{\sN'_*}(W)$ gives us  
$\Omega_*^{\sN'_*(1)}(X)$.   The maps $\bar{\vartheta}^{(1)}$,
$\bar{\vartheta}(X)$ and $\bar{\vartheta}$ are all induced by the classifying
map $\vartheta_{\MGL'_*/\sN'_*}$. The top row is a complex and the bottom row is
exact; this latter fact follows from the surjectivity assumption in 
definition~\ref{def:GeomLE}(ii). The map $\bar{\vartheta}$ is an isomorphism by
part (i) of definition~\ref{def:GeomLE} and $\bar\vartheta^{(1)}$ is an
isomorphism by induction on $d$. To show that $\bar\vartheta(X)$ is an
isomorphism, it suffices to show that the identity map on $\oplus_\eta
\sN'_{*-d+1}\otimes k(\eta)^\times$ extends diagram~\eqref{eqn:Diag3} to a
commutative diagram.

To see this, we note that the map $div_\sN$ is defined by composing the boundary
map
\[
\del:\oplus_{\eta\in X_{(d)}}\sN'_{2*+1, *}(\eta)\to \sN_{2*,*}^{\prime
(1)}(X)
\]
with the sum of the product maps $\MGL'_{2d-1, d-1}(\eta)\otimes \sN'_{*-d+1}(k)\to \sN'_{2*+1, *}(\eta)$ and the canonical map
$t_\MGL(\eta):k(\eta)^\times \to \MGL^{1,1}(\eta)=\MGL'_{2d-1, d-1}(\eta)$ (see
\cite[remark 1.5]{LevineComp}). For $\MGL'$, we have the similarly defined map
\[
div_\MGL:\oplus_{\eta\in X_{(d)}} k(\eta)^\times\otimes  \L_{*-d+1} \to
\MGL_{2*,*}^{\prime (1)}(X),
\]
after replacing $\MGL'_{*-d+1}(k)$ with $\L_{*-d+1}$ via the classifying map
$\L_*\to \MGL'_*(k)$. We have as well the
 commutative diagram (see \cite[(5.4)]{LevineComp})
\[
\xymatrixcolsep{25pt}
\xymatrix{
\oplus_{\eta\in X_{(d)}}k(\eta)^\times\otimes 
\L_{*-d+1}\ar[r]^-{\Div}\ar@{=}[d]&\Omega_*^{(1)}(X)\ar[d]^{\vartheta_\MGL^{(1)}
}\\
\oplus_{\eta\in X_{(d)}}k(\eta)^\times\otimes  \L_{*-d+1}\ar[r]_-{div_\MGL} 
&\MGL_{2*,*}^{\prime (1)}(X),
}
\]
 which after applying $-\otimes_{\L_*}\sN'_*$ gives us the commutative diagram
\begin{equation}\label{eqn:HM1}
\xymatrixcolsep{25pt}
\xymatrix{
\oplus_{\eta\in X_{(d)}}k(\eta)^\times\otimes 
\sN_{*-d+1}\ar[r]^-{\Div_\sN}\ar@{=}[d]&\Omega_*^{\sN_*'
(1)}(X)\ar[d]^{\vartheta_\MGL^{(1)}\otimes\id}\\
\oplus_{\eta\in X_{(d)}}k(\eta)^\times\otimes \sN_{*-d+1}\ar[r]_-{\overline{div}_\MGL} 
&\MGL_{2*,*}^{\prime (1)}(X)\otimes_{\L_*}\sN'_*
}
\end{equation}

The Leibniz rule for $\partial$ gives us the commutative diagram 
\begin{equation}\label{eqn:HM2}
\xymatrixcolsep{25pt}
\xymatrix{
\oplus_{\eta\in X_{(d)}}k(\eta)^\times\otimes \sN_{*-d+1}\ar[r]^-{\overline{div}_\MGL}
\ar@{=}[d]
&\MGL_{2*,*}^{\prime (1)}(X)\otimes_{\L_*}\sN'_*\ar[d]^{\cup}\\
\oplus_{\eta\in X_{(d)}}k(\eta)^\times\otimes \sN'_{*-d+1} \ar[r]_-{div_\sN}
&\sN_{2*,*}^{\prime (1)}(X);
}
\end{equation}
combining diagrams \eqref{eqn:HM1} and \eqref{eqn:HM2} yields the desired
commutativity.
\end{proof}

\section{Applications to quotients of $\MGL$}\label{sec:Apps}

We return to our discussion of quotients of $\MGL_p$ and their localizations. We
select a system of polynomial generators for the Lazard ring, $\L_*\cong \Z[x_1,
x_2,\ldots]$, $\deg\, x_i=i$. Let $\sS\subset \N$,  $\sS^c$ its complement and
 let $\Z[\sS^c]$ denote the graded polynomial ring on the $x_i$,
$i\in\sS^c$, $\deg x_i=i$. Let 
$\sS_0\subset \Z[\sS^c]$ be a collection of homogeneous elements, $\sS_0=\{z_j\in  \Z[\sS^c]_{e_j}\}$, and let  $\Z[\sS^c][\sS_0^{-1}]$ denote the localization of $\Z[\sS^c]$ with
respect to $\sS_0$.

We consider   a quotient spectrum
$\MGL_p/(\sS):=\MGL_p/(\{x_i\ |\ i\in \sS\})$ or an integral version
$\MGL/(\sS):=\MGL/(\{x_i\ |\ i\in \sS\})$. We consider as well the localizations
\begin{gather*}
\MGL_p/(\sS)[\sS_0^{-1}]:=\MGL_p/(\sS)[\{z_j^{-1}\ |\ z_j\in \sS_0\}],\\
\MGL/(\sS)[\sS_0^{-1}]:=\MGL/(\sS)[\{z_j^{-1}\ |\ z_j\in \sS_0\}].
\end{gather*}
and the mod $p$ version
\[
\MGL/(\sS,p)[\sS_0^{-1}]:=\MGL_p/(\sS)[\sS_0^{-1}]/p
\]

\begin{proposition}\label{prop:GeomLEQuot} Let $p$ be a prime, and let $S=\Spec
k$, $k$ a perfect field with exponential characteristic prime to $p$. Let $\sS$ be a subset
of $\N$ and $\sS_0$ a set of homogeneous elements of $\Z[\sS^c]$. Then the spectra 
 $\MGL_p/(\sS)[\sS_0^{-1}]$ and
$\MGL_p/(\sS,p)[\sS_0^{-1}]$ are geometrically Landweber exact. In case $\Char k=0$, 
$\MGL/(\sS)[\sS_0^{-1}]$ is geometrically Landweber exact.
\end{proposition}

\begin{proof} We discuss the cases $\MGL_p/(\sS)[\sS_0^{-1}]$ and
$\MGL_p/(\sS,p)[\sS_0^{-1}]$; the case  of $\MGL/(\sS)[\sS_0^{-1}]$ is exactly the
same. 

Let $A$ be a finitely generated abelian group and let $\eta$ be a point in some
$X\in \Sm/k$.  Then  the motivic cohomology $H^*(\eta, A(*))$ satisfies
\[
H^{2r}(\eta, A(r))= H^{2r+1}(\eta, A(r+1))=0
\]
for $r\neq0$, 
\[
H^0(\eta, A(0))=A,\quad H^1(\eta,A(1))=k(\eta)^\times\otimes_\Z A.
\]
We consider the  slice spectral sequences
\[
E_2^{p,q}(n):=H^{p-q}(\eta, \Z(n-q))\otimes
\Z[\sS^c][\sS_0^{-1}]_{-q}\Rightarrow (\MGL_p/(\sS)[\sS_0^{-1}])^{p+q,n}(\eta)
\]
and
\[
E_2^{p,q}(n):=H^{p-q}(\eta, \Z/p(n-q))\otimes
\Z[\sS^c][\sS_0^{-1}]_{-q}\Rightarrow
(\MGL_p/(\sS,p)[\sS_0^{-1}])^{p+q,n}(\eta)
\]
given by proposition~\ref{prop:Conv2}.  As in the proof of theorem~\ref{slice spectral sequence for BP and BP<n>}, $\MGL^{2n+a,n}_p(\eta)=0$ for $a>0$ and $n\in\Z$, and thus by remark~\ref{rem:Converge}, the convergence hypotheses in proposition~\ref{prop:Conv2} are satisfied. Thus, these spectral sequences are strongly convergent. As discussed in the proof of \cite[proposition 3.8]{LevineExact},   the only
non-zero $E_2$ term contributing to 
$(\MGL_p/(\sS)[\sS_0^{-1}])^{2n,n}(\eta)$ or to
$(\MGL_p/(\sS,p)[\sS_0^{-1}])^{2n,n}(\eta)$ is $E_2^{n,n}(n)$, the 
only non-zero $E_2$ term contributing to 
$(\MGL_p/(\sS)[\sS_0^{-1}])^{2n-1,n}(\eta)$ or  contributing to
$(\MGL_p/(\sS,p)[\sS_0^{-1}])^{2n-1,n}(\eta)$ is $E_2^{n,n-1}(n)$, and all
differentials entering or leaving these terms are zero.

This gives us isomorphisms
\begin{align*}
&(\MGL_p/(\sS)[\sS_0^{-1}])^{2n,n}(\eta)\cong \Z_{(p)}[\sS^c][\sS_0^{-1}]_n\\
&(\MGL_p/(\sS,p)[\sS_0^{-1}])^{2n,n}(\eta)\cong \Z/(p)[\sS^c][\sS_0^{-1}]_n\\
&(\MGL_p/(\sS)[\sS_0^{-1}])^{2n-1,n}(\eta)\cong
\Z_{(p)}[\sS^c][\sS_0^{-1}]_n\otimes k(\eta)^\times\\
&(\MGL_p/(\sS,p)[\sS_0^{-1}])^{2n,n}(\eta)\cong
\Z/(p)[\sS^c][\sS_0^{-1}]_n\otimes k(\eta)^\times
\end{align*}
from which it easily follows that $\MGL_p/(\sS)[\sS_0^{-1}]$ and
$\MGL_p/(\sS,p)[\sS_0^{-1}]$ are geometrically Landweber exact.
\end{proof} 

\begin{corollary} Let $S=\Spec k$, $k$ a field of characteristic zero.  Fix a
prime $p$ and let $\sN=$ $\MGL/(\sS)[\sS_0^{-1}]$, $\MGL_p/(\sS)[\sS_0^{-1}]$ or
$\MGL_p/(\sS,p)[\sS_0^{-1}]$,  let $(\sN', \sN)$ be the associated $(\MGL',
\MGL)$-module and $\sN'_*$ the geometric part of $\sN'_{**}$. Then the
classifying map
\[
\vartheta_{\sN'_*(k)}: \Omega_*^{\sN'_*(k)}\to \sN'_*
\]
is an isomorphism of $\Omega_*$-modules.
\end{corollary}

This follows directly from proposition~\ref{prop:GeomLEQuot}. As immediate
consequence, we have

\begin{corollary} Let $S=\Spec k$, $k$ a field of characteristic zero. Fix a
prime $p$ and let $\sN=$ $\BP$, $\BP\langle n\rangle$, $E(n)$,  $k(n)$ or
$K(n)$, let $(\sN', \sN)$ be the associated $(\MGL', \MGL)$-module and $\sN'_*$
the geometric part of $\sN'_{**}$. Then the classifying map
\[
\vartheta_{\sN'_*(k)}: \Omega_*^{\sN'_*(k)}\to \sN'_*
\]
is an isomorphism of  $\Omega_*$-modules. In case $\sN=\BP$ or $E(n)$,
$\vartheta_{\sN'_*(k)}$ is compatible with external products.
\end{corollary}

\begin{remark}  Suppose that the theory with supports $\sN^{2*,*}$ has products
and a unit, compatible with its $\MGL^{2*,*}$-module structure. Then by
remark~\ref{rem:products}, the classifying map $\vartheta_{\sN'_*(k)}$ is also
compatible with products.

In the case of a quotient $\sE$ of $\MGL$ or $\MGL_p$ by subset $\{x_i:i\in I\}$
of the set of polynomial generators, the vanishing of $\MGL^{2r+s, r}(k)$ for
$s>0$ shows that $\sE^{2*,*}(k)=\MGL^{2*,*}(k)/(\{x_i:i\in I\})$, which has the
evident ring structure induced by the natural $\MGL^{2*,*}(k)$-module structure.
Thus, the rational theory $\Omega_*^{\sE_*(k)}$ has a canonical structure of an
oriented Borel-Moore homology theory on $\Sch/k$; the same holds for $\sE$ a
localization of this type of quotient. The fact that the classifying
homomorphism $\vartheta_{\sE}:\Omega_*^{\sE'_*(k)}\to \sE_*'$ is an isomorphism
induces on $\sE'_*$ the structure of an oriented Borel-Moore homology theory on
$\Sch/k$; it appears to be unknown if this arises from a multiplicative
structure on the spectrum level.
\end{remark}

\end{document}